\documentclass[reqno,oneside,12pt]{amsart}

\usepackage[T1]{fontenc}
\usepackage{times,mathptm,dsfont}
\usepackage{amssymb,epsfig,verbatim,xypic}

\usepackage{mathtools}
\usepackage{color}
\theoremstyle{plain}

\newtheorem{thm}{Theorem}[section]

\newtheorem{cor}[thm]{Corollary}
\newtheorem{pro}[thm]{Proposition}
\newtheorem{lem}[thm]{Lemma}
\newtheorem{proposition-principale}[thm]{Proposition principale}
\newtheorem{thm-principal}{Main Theorem}
\newtheorem{defi}[thm]{Definition}

\theoremstyle{definition}
\newtheorem{que}{Question}[section]
\newtheorem{eg}[thm]{Example}
\newtheorem{rem}[thm]{Remark}

\newtheorem*{thm-B}{Theorem B}
\newtheorem*{defi-G}{Definition}
\newtheorem*{thm-C}{Classification Theorem}
\newtheorem*{thm-A}{Theorem A}
\newtheorem*{thm-BB}{Theorem B'}

\def\C{\mathbf{C}}
\def\R{\mathbf{R}}
\def\Q{\mathbf{Q}}

\def\Z{\mathbf{Z}}

\def\bfk{\mathbf{k}}
\def\oZ{{\overline{\mathbf{Z}}}}

\def\P{\mathbb{P}}
\def\A{\mathbb{A}}
\def\F{{\mathbb{F}}}
\def\SS{\mathbb{S}}

\def\m{{\mathfrak{m}}}

\def\U{{\mathcal{U}}}
\def\V{{\mathcal{V}}}
\def\Hirz{{\mathbb{F}}}

\def\Hy{{\sf{Hyp}}}
\def\Hyp{{\sf{\tilde{Hyp}}}}
\def\Bij{\sf{Perm}}

\def\Dual{{\mathcal{D}}}

\def\exc{{\sf{exc}}}

\def\G{{\mathcal{G}}}
\def\Tho{G_{\mathrm{Far}}}
\def\Thom{G_{\mathrm{Dya}}}

\def\W{{\mathcal{W}}}

\def\B{\mathcal{B}}

\def\Aut{{\sf{Aut}}}
\def\Psaut{{\sf{Psaut}}}
\def\Bir{{\sf{Bir}}}

\def\dv{{\mathrm{DV}}}

\newcommand{\FW}{(FW)}
\def\PW{{{(PW)}}}
\def\FH{{{(FH)}}}

\def\Jac{{\sf{Jac}}}

\def\Far{{\mathrm{Far}}}
\def\Fib{{\mathrm{Fib}}}
\def\Dya{{\mathrm{Dyad}}}

\def\NS{{\mathsf{NS}}}

\def\PGL{{\sf{PGL}}\,}

\def\GL{{\sf{GL}}\,}
\def\SO{{\sf{SO}}\,}
\def\O{{\sf{O}}\,}
\def\Iso{{\sf{Iso}}\,}

\def\SL{{\sf{SL}}\,}

\newcommand{\Id}{{\rm Id}}

\def\Ind{{\mathrm{Ind}}}

\addtocounter{section}{0}             
\numberwithin{equation}{section}       

\begin{document}

\setlength{\baselineskip}{0.56cm}        
%
%

\title[Commensurating actions of birational groups]
{Commensurating actions of birational groups and groups of pseudo-automorphisms}
\date{June 29, 2019}
\author{Serge Cantat and Yves de Cornulier}
\address{IRMAR (UMR 6625 du CNRS)\\ 
Universit{\'e} de Rennes 1 
\\ France}
\address{CNRS and Univ Lyon, Univ Claude Bernard Lyon 1, Institut Camille Jordan, 43 blvd. du 11 novembre 1918, F-69622 Villeurbanne}
\email{serge.cantat@univ-rennes1.fr}
\email{cornulier@math.univ-lyon1.fr}




%
%

%
%

%
%

\begin{abstract} 
Pseudo-automorphisms are birational transformations acting as regular automorphisms in codimension $1$. We import
ideas from geometric group theory to prove that a group of birational 
transformations that satisfies a fixed point property on {\sc{cat}}$(0)$ cubical complexes, for example a discrete countable group with Kazhdan Property (T), 
is birationally conjugate to a group acting by pseudo-automorphisms
on some non-empty Zariski-open subset. We apply this argument to classify groups of birational transformations of surfaces 
with this fixed point property up to birational conjugacy. 
\end{abstract}

\maketitle

\setcounter{tocdepth}{1}

\section{Introduction}

\subsection{Birational transformations and pseudo-automorphisms}
Let $X$ be a quasi-projective variety, over an algebraically closed field $\bfk$. Denote by $\Bir(X)$ the group of birational transformations of $X$ 
and by $\Aut(X)$  the subgroup of (regular) automorphisms of $X$. For the affine space of dimension $n$, 
automorphisms are invertible transformations $f\colon \A^n_\bfk \to \A^n_\bfk$ such that both $f$
and $f^{-1}$ are defined by polynomial formulas in affine coordinates:
\[
f(x_1, \ldots, x_n)=(f_1, \ldots, f_n), \; \; f^{-1}(x_1, \ldots, x_n)=(g_1, \ldots, g_n)
\]
with $f_i$, $g_i \in \bfk[x_1, \ldots, x_n]$. Similarly, birational transformations 
of $\A^n_\bfk$ are given by rational formulas, i.e.\ $f_i$, $g_i\in \bfk(x_1, \ldots, x_n)$.

Birational transformations may contract hypersurfaces.  {\bf{Pseu\-do-automorphisms}} are birational 
transformations that act as automorphisms in codimension $1$. Precisely, a birational transformation 
$f\colon X \dasharrow X$ is a pseudo-automorphism if there exist Zariski-open subsets $\U$
and $\V$ in $X$ such that $X\smallsetminus \U$ and $X\smallsetminus \V$ have codimension $\geq 2$ 
and $f$ induces an isomorphism from $\U$ to $\V$. The pseudo-automorphisms of $X$ 
form a group, which we denote by $\Psaut(X)$. For instance, all birational transformations of Calabi-Yau 
manifolds are pseudo-automorphisms; and there are examples of such manifolds
for which $\Psaut(X)$ is infinite while $\Aut(X)$ is trivial (see \cite{Cantat-Oguiso:2015}). Pseudo-automorphisms
are studied in Section~\ref{par:pseudo-automorphisms}.

\begin{defi}\label{d_psr}
Let $\Gamma\subset\Bir(X)$ be a group of birational transformations of an irreducible projective variety $X$. We say that $\Gamma$ is {\bf pseudo-regularizable} if there exists a triple $(Y,\U,\varphi)$ where
\begin{enumerate}
\item $Y$ is a projective variety and $\varphi\colon Y\dasharrow X$ is a birational map;
\item $\U$ is a dense Zariski open subset of $Y$;
\item $\varphi^{-1}\circ \Gamma \circ \varphi$ yields an action of $\Gamma$ by pseudo-automorphisms on $\U$.
\end{enumerate}
More generally if $\alpha:\Gamma\to\Bir(X)$ is a homomorphism, we say that 
it is pseudo-regularizable if $\alpha(\Gamma)$ is pseudo-regularizable.
\end{defi}

One goal of this article is to use rigidity properties of commensurating actions, a purely group-theoretic concept, to show that many group actions are pseudo-regularizable. In particular, we exhibit a class of groups for which all actions by birational transformations on projective varieties  are pseudo-regularizable.

\subsection{Property {\FW}}\label{par:Intro-FW}
The class of groups we shall be mainly interested in is characterized by a fixed point property appearing
in several related situations, for instance for actions on {\sc{cat}}$(0)$ cubical complexes.
Here, we adopt the viewpoint of commensurated subsets.
Let $\Gamma$ be a group, and $\Gamma \times S\to S$ an action of $\Gamma$ on a set $S$. Let
$A$ be a subset of $S$. One says that $\Gamma$ {\bf{commensurates}} $A$ if the symmetric difference
\[
\gamma(A)\triangle A= \left( \gamma(A)\smallsetminus A \right) \, \cup\,  \left( A\smallsetminus \gamma(A) \right)
\] 
is finite for every element $\gamma$ of $\Gamma$. One says that $\Gamma$
{\bf{transfixes}} $A$ if there is a subset $B$ of $S$ such that $A\triangle B$ is finite and $B$ is $\Gamma$-invariant:
$\gamma(B)=B,$ for every $\gamma$ in $\Gamma$.

A group $\Gamma$ has {\bf{Property {\FW}}} if, given any action of $\Gamma$ on any set $S$, all
commensurated subsets of $S$ are automatically transfixed. For instance, the cyclic group $(\Z,+)$ 
acts on itself by translation, this action commensurates $\Z_+$ but does not transfix it, hence $\Z$ 
does not have Property~{\FW}; more generally, Property~{\FW} is not satisfied by non-trivial free groups.
To get examples, recall that a countable group $\Gamma$ has Kazhdan {\bf{Property (T)}} if every affine isometric 
action of $\Gamma$ on a Hilbert space fixes a point:  Property (T) implies~{\FW}, so that
all lattices in higher rank simple Lie groups have Property~{\FW}, for instance $\SL_m(\Z)$  when $m\geq 3$ (see \cite{Cornulier:Survey-FW} and
Section~\ref{par:(FW)-(T)}).
The group $\SL_2(\Z[\sqrt{5}])$ also has Property~{\FW} without satisfying Property~(T) (see \cite{Cornulier:Survey-FW}).

Property {\FW} is discussed in Section~\ref{par:FW}. 
Let us mention that among its various characterizations, one is: 
every combinatorial action of $\Gamma$ on a {\sc{cat}}(0) cube complex fixes 
some cube. Another, for $\Gamma$ finitely generated, is that all its infinite connected Schreier graphs are one-ended
(see \cite{Cornulier:Survey-FW}). 

\subsection{Pseudo-regularizations}
Let $X$ be a projective variety. The group $\Bir(X)$ does not really act on $X$, because there
are indeterminacy points; it does not act on the set of hypersurfaces either, because some of
them may be contracted. As we shall see, one can introduce the set $\Hyp(X)$ of all irreducible and reduced hypersurfaces in all birational models 
$X'\dasharrow X$ (up to a natural identification). Then, 
there is a natural action of the group $\Bir(X)$ on this set, given by strict transforms  of hypersurfaces. 
Indeed, the rigorous construction of this action follows naturally from the action on the set of divisorial valuations.
Since this action commensurates the subset $\Hy(X)$ of hypersurfaces of $X$, this construction leads to the following result.

\begin{thm-A}
{\sl{Let $X$ be a projective variety over an algebraically closed field. Let
$\Gamma$ be a subgroup of $\Bir(X)$. If $\Gamma$ has Property {\FW}, then $\Gamma$ is pseudo-regularizable.}}
\end{thm-A}

There is also a relative version of Property {\FW}  for pairs of groups $\Lambda \leq \Gamma$, which leads to 
a similar pseudo-regularization theorem for the subgroup $\Lambda$: this is discussed in Section~\ref{par:distorsion}, 
with applications to distorted birational transformations. 

\begin{rem}\label{rem:intro-extreme-cases}
Theorem~A provides a triple $(Y,\U,\varphi)$ such that $\varphi$ conjugates $\Gamma$ to a group of pseudo-automorphisms on the open subset $\U\subset Y$.
There are two extreme cases for the pair $(Y,\U)$ depending on the size of the boundary $Y\smallsetminus \U$. 
If this boundary is empty, $\Gamma$ acts by pseudo-automorphisms on a projective variety $Y$. If it is ample, its complement 
$\U$ is an affine variety; if $\U$ is smooth (or locally factorial) then $\Gamma$ actually acts by regular automorphisms on $\U$ (see Section~\ref{par:affine}). 
Thus, in the study of groups of birational transformations, {\sl{pseudo-automorphisms of projective varieties and regular
automorphisms of affine varieties deserve specific attention}}. 
\end{rem}

\subsection{Classification in dimension $2$}
In dimension $2$, pseudo-automorphisms do not differ much from automorphisms; for instance, $\Psaut(X)$
coincides with $\Aut(X)$ if $X$ is a smooth projective surface. Thus, for groups with Property {\FW}, Theorem~A can be used
to reduce the study of birational transformations to the study of automorphisms of quasi-projective
surfaces. Combining results of Danilov
and Gizatullin on automorphisms of affine surfaces with a theorem of Farley on groups of piecewise affine transformations 
of the circle, we prove the following theorem. 

\begin{thm-B}
{\sl{Let $X$ be a smooth, projective, and irreducible surface, over an algebraically closed field. Let
$\Gamma$ be an infinite subgroup of $\Bir(X)$. If $\Gamma$ has Property {\FW}, there 
is a birational map $\varphi\colon Y\dasharrow X$ 
such that 
\begin{enumerate}
\item $Y$ is the projective plane $\P^2$, a Hirzebruch surface $\Hirz_m$ with $m\geq 1$, or the product of a curve $C$ by the projective line $\P^1$. If the characteristic of the field is positive, $Y$ is the projective plane $\P^2_\bfk$.
\item $\varphi^{-1}\circ \Gamma\circ \varphi$ is contained in $\Aut(Y)$. 
\end{enumerate}
}}
\end{thm-B}

\begin{rem}
There is an infinite subgroup  of $\Aut(Y)$ with Property~{\FW} for all surfaces $Y$ of Assertion~(1). 
Namely, if the algebraically closed field $\mathbf{k}$ has characteristic zero, $\Aut(Y)$ contains $\PGL_2(\mathbf{k})$ or the quotient of $\GL_2(\bfk)$ 
by a central cyclic subgroup in case $Y$ is a Hirzebruch surface. Thus, there is a morphism $\SL_2(\Z[\sqrt{5}])\to \Aut(Y)$ with finite kernel  and, as mentioned in Section~\ref{par:Intro-FW}, $\SL_2(\Z[\sqrt{5}])$ has Property~{\FW}.
 In characteristic $p>0$, the only case is that of $\P^2_{\bfk}$,  whose automorphism group contains the group $\mathrm{PSL}_3((\mathbf{F}_p[t]))$, which has Kazhdan's Property~(T).
\end{rem}

\begin{rem} The group $\Aut(Y)$ has finitely many connected components for all surfaces $Y$ 
of Assertion~(1) in Theorem~B. Thus, changing $\Gamma$
into a finite index subgroup, one gets a subgroup of $\Aut(Y)^0$; here $\Aut(Y)^0$ denotes the 
connected component of the identity: this is an algebraic group, acting algebraically on $Y$.\end{rem}

\begin{eg}\label{applx}
Groups with Kazhdan Property (T) satisfy Property {\FW} (see Section~\ref{par:FW}). Also, 
if $Y$ is a Hirzebruch surface or a product $C\times \P^1$ for some curve $C$, then $\Aut(Y)$
does not contain any group with Property (T), because the group $\PGL_2(\bfk)$ does not contain 
such a group. Thus, Theorem~B extends Theorem~A of \cite{Cantat:Annals}, at least in the projective case, and
 the present article offers a new proof of that result. 
Theorem~B can also be applied to the group $\SL_2(\Z[\sqrt{d}])$ when the integer $d\ge 2$ is not a perfect square:
every action of this group on a projective surface
by birational transformations is conjugate to an action by regular automorphisms on 
$\P^2_\bfk$, the product of a curve $C$ by the projective line $\P^1_\bfk$, or a Hirzebruch surface. 
Theorem~\ref{tSL2} provides a more precise result, based on Theorem~B and Margulis' superrigidity theorem.
\end{eg}

\begin{rem} Let $X$ be a normal projective variety. 
One can ask whether $\Bir(X)$ transfixes $\Hy(X)$, or equivalently is pseudo-regularizable (see Theorem~\ref{thm:FW-pseudo-regularization}). 
For surfaces, this holds precisely when $X$ is not birationally equivalent to the product of the projective line with a curve. 
See \S\ref{surf_birt} for more precise results.
\end{rem}

\subsection{Acknowledgement}

This work benefited from interesting discussions with J\'er\'emy Blanc, Vincent Guirardel, and Christian Urech. 
We are grateful to the referees, for pointing out a gap in a proof and suggesting many improvements.

\section{Pseudo-automorphisms}\label{par:pseudo-automorphisms}

This preliminary section introduces useful notation for birational transformations and pseudo-automorphisms, and presents a few basic results.

\subsection{Birational transformations}
Let $X$ and $Y$ be two irreducible and reduced algebraic varieties over an algebraically closed field $\bfk$. 
Let $f\colon X\dasharrow Y$ be a birational map. Choose dense Zariski open subsets $U\subset X$ and $V\subset Y$ such that 
$f$ induces an isomorphism $f_{U,V}\colon U\to V$. Then the graph $\mathfrak{G}_f$ of $f$ is defined as 
the Zariski closure of $\{(x,f_{U,V}(x)):x\in U\}$ in $X\times Y$; it does not depend on the choice of $U$ and $V$. 
The graph $\mathfrak{G}_f$ is an irreducible variety; both projections 
\[
u\colon \mathfrak{G}_f  \to X\quad  {\text{and}} \quad 
v\colon \mathfrak{G}_f \to Y
\] 
are birational morphisms and $f=v\circ u^{-1}$. 

We shall denote by  $\Ind(f)$ the indeterminacy set of the birational map $f$.

\begin{thm}[Theorem 2.17 in \cite{iitaka1982algebraic}]\label{normal2}
Let $f:X\dashrightarrow Y$ be a rational map, with $X$ a normal variety and $Y$ a projective variety. 
Then the indeterminacy set of $f$ has codimension $\ge 2$.
\end{thm}

\begin{eg} The transformation of the affine plane $(x,y)\mapsto (x,y/x)$ is 
birational, and its indeterminacy locus is the line $\{x=0\}$: this set of codimension $1$ is mapped ``to infinity''. If the affine 
plane is compactified by the projective plane, the transformation becomes $[x:y:z]\mapsto [x^2:yz:xz]$, with two indeterminacy points. 
\end{eg}

The {\bf{total transform}} of a closed subset $Z\subset X$ is denoted by  $f_*(Z)$; by definition, $f_*(Z)=v(u^{-1}(Z))$. 
If $Z$ is   irreducible and is not contained in $\Ind(f)$, we denote
by $f_\circ(Z)$ its {\bf{strict transform}}, defined as the Zariski closure of $f(Z\smallsetminus \Ind(f))$. We say that an irreducible hypersurface
$W\subset X$ is contracted if it is not contained in the indeterminacy set and the codimension of its strict transform is larger than $1$; 
the {\bf{exceptional divisor}} of $f$ is the union of all contracted hypersurfaces. 

We say that $f$ is a {\bf{local isomorphism}} near a point $x\in X$ if there are open sets $U\subset X$ 
and $V\subset Y$ such that $U$ contains $x$ and $f$ induces an isomorphism from $U$ to $V$.
The {\bf{exceptional set}} of $f$ is the subset of $X$ along which $f$
is not a local isomorphism; this set is Zariski closed, and is made of three
parts: the indeterminacy locus, the exceptional divisor, and a residual part of codimension~$\geq 2$.

\subsection{Pseudo-isomorphisms}
 A birational map $f\colon X \dasharrow Y$ is a {\bf{pseudo-isomor\-phism}} if one 
can find Zariski open subsets $\U\subset X$ and $\V \subset Y$ such that 
\begin{itemize}
\item[(i)] $f$ realizes a regular isomorphism 
from $\U$ to $\V$ and 
\item[(ii)] $X\smallsetminus \U$ and $Y\smallsetminus \V$ have codimension $\geq 2$. 
\end{itemize}

Pseudo-isomorphisms from $X$ to itself are called {\bf{pseudo-automorphisms}} (see \S~\ref{par:Intro-FW}). 
{\sl{The set of pseudo-automorphisms of $X$ is a subgroup $\Psaut(X)$ of~$\Bir(X)$}}. 

\begin{eg} Start with the
standard birational involution $\sigma_n\colon \P^n_\bfk\dasharrow \P^n_\bfk$ which is defined in homogeneous 
coordinates by $\sigma_n[x_0:\ldots : x_n]=[x_0^{-1}:\ldots : x_n^{-1}].$
Blow-up the $(n+1)$ vertices of the simplex $\Delta_n=\{[x_0:\ldots : x_n ]; \; \prod x_i=0\}$; this provides a smooth
rational variety $X_n$ together with a birational morphism $\pi\colon X_n\to \P^n_\bfk$. Then,  $\pi^{-1}\circ \sigma_n\circ \pi$ 
 is a pseudo-automorphism  of $X_n$, and is an automorphism if $n\leq 2$. 
 \end{eg}

\begin{pro}\label{pro:pseudo-isomorphism}
Let $f\colon X\dasharrow Y$ be a birational map between two (irreducible, reduced) normal algebraic varieties.  Then, the following properties are equivalent: 
\begin{enumerate}
\item The birational maps $f$ and $f^{-1}$ do not contract any hypersurface, and their
indeterminacy sets have codimension $\geq 2$ in $X$ and $Y$ respectively.
\item The birational map $f$ is a pseudo-isomorphism from $X$ to $Y$. 
\end{enumerate}
\end{pro}

\begin{proof} Denote by $g$ the inverse of $f$. The second assertion implies the first because any hypersurface intersects the complement of
every closed subset of codimension~$\geq 2$. Let us prove that the first assertion implies the second. 
Let $\U_0\subset X$ (resp. $\V_0\subset Y$) be the complement of the singular locus of $X$ (resp. $Y$) and the indeterminacy locus of $f$ (resp. $g$).
Let $\U$ be the pre-image of $\V_0$ by the birational map $f_{\U_0}\colon \U_0\dasharrow \V_0$; the complement of $\U$ in $\U_0$, and therefore
in $X$ too, has codimension $\geq 2$ because the codimension of $Y\setminus \V_0$ is at least $2$ and $f$ does not contract any hypersurface. 
Define $\V\subset \V_0$ to be the pre-image of $\U$ by $g$ (restricted to $\V_0$); the codimension of $Y\setminus \V$ is also $\geq 2$. Then, the restriction $f_{\U}\colon \U\dasharrow \V$
is a regular isomorphism, with inverse $g_{\V}\colon \V\dasharrow \U$. 
\end{proof}

\begin{eg}
Let $X$ be a smooth projective variety with trivial canonical bundle $K_X$. Let $\Omega$ be a non-vanishing section of $K_X$, and let
$f$ be a birational transformation of~$X$. Then, $f^*\Omega$ extends from $X\smallsetminus \Ind(f)$ to $X$ and determines a new section of
$K_X$; this section does not vanish identically because $f$ is dominant, hence it does not vanish at all because $K_X$ is trivial. As
a consequence, $f$ does not  contract any hypersurface, because otherwise $f^*\Omega$ would vanish along this hypersurface. 
Since $X$ is projective, the codimension of $\Ind(f)$ is $\geq 2$ (Theorem~\ref{normal2}). 
Thus, $f$ is a pseudo-automorphism of $X$, and  $\Bir(X)=\Psaut(X)$. 
We refer to \cite{Cantat-Oguiso:2015, Fryers:preprint} for families of Calabi-Yau varieties with an infinite group of pseudo-automorphisms.
\end{eg}

\subsection{Projective varieties}

\subsubsection{Smooth varieties}
Assume that $X$ and $Y$ are smooth.
The {\bf{Jacobian determinant}} $\Jac(f)(x)$ is defined in local coordinates 
as the determinant of the differential $df_x$; the rational function $\Jac(f)$ 
depends on the chosen coordinates (on $X$ and $Y$), but its zero locus does not. 
The zeroes of $\Jac(f)$ form a hypersurface of  $X\setminus\Ind(f)$; the zero locus of $\Jac(f)$ will be defined as the
Zariski closure of this hypersurface in $X$.

\begin{pro}\label{pro:pi-smooth}
Let $f\colon X\dasharrow Y$ be a birational transformation between two smooth varieties. Assume that $\Ind(f)$
and $\Ind(f^{-1})$ have codimension $\geq 2$. The following properties are equivalent.  
\begin{enumerate}
\item The Jacobian determinants of $f$ and $f^{-1}$ do not vanish.
\item For every  $q\in X\smallsetminus \Ind(f)$, $f$ is an isomorphism from a 
neighborhood of $q$ to a neighborhood of $f(q)$, and the same holds for $f^{-1}$.
\item The birational map $f$ is a pseudo-isomorphism from $X$ to $Y$. 
\end{enumerate}
\end{pro}

\begin{proof}
Denote by $g$ the inverse of $f$. If the Jacobian determinant of $f$ vanishes at some point of $X\smallsetminus \Ind(f)$, then it 
vanishes along a hypersurface $V\subset X$. If (1) is satisfied, then $f$ does not contract any positive
dimensional subset of $X\smallsetminus\Ind(f)$: $f$ is a quasi-finite map from $X\smallsetminus \Ind(f)$ to its image, and so 
is $g$. Zariski's main theorem implies that $f$ realizes an isomorphism from $X\smallsetminus \Ind(f)$ to 
$Y\smallsetminus \Ind(g)$ (see \cite{Milne:BookAG}, Prop. 8.57). Thus, (1) implies (2) and (3). 
Since (3) implies (1), this concludes the proof.   \end{proof}

\begin{pro}[see \cite{Bedford-Kim:2014}] Let $f\colon X\dasharrow Y$ be a pseudo-isomorphism between two 
smooth projective varieties. Then
\begin{enumerate}
\item the total transform of $\Ind(f)$ by $f$ is equal to $\Ind(f^{-1})$;
\item $f$ has no isolated indeterminacy point; 
\item if $\dim(X)=2$, then $f$ is a regular isomorphism. 
\end{enumerate}
\end{pro}

\begin{proof} Since $X$ and $Y$ are projective, $\Ind(f)$ and $\Ind(f^{-1})$ have codimension $\geq 2$: we can 
apply Propositions~\ref{pro:pseudo-isomorphism} and~\ref{pro:pi-smooth}. 
Let $p\in X$ be an indeterminacy point of the pseudo-isomorphism $f\colon X\dasharrow Y$.
Then $f^{-1}$ contracts a subset $C\subset Y$ of positive dimension  on $p$. Since $f$ and
$f^{-1}$ are local isomorphisms on the complement of their indeterminacy sets, $C$ is contained
in $\Ind(f^{-1})$. The total transform of a point $q\in C$ by $f^{-1}$ is a connected subset of $X$ 
that contains $p$ and has dimension $\geq 1$. This set $D_q$ is contained in $\Ind(f)$ because
$f$ is a local isomorphism on the complement of $\Ind(f)$; since $p\in D_q\subset \Ind(f)$, $p$ is
not an isolated indeterminacy point. This proves Assertions~(1) and~(2).
The third assertion follows from the second one because indeterminacy sets of birational transformations
of projective surfaces are finite sets. 
\end{proof}

\subsubsection{Divisors and N\'eron-Severi group} Let $W$ be a hypersurface of $X$, and let $f\colon X \dasharrow Y$ be a pseudo-isomorphism.
The divisorial part of the total transform $f_*(W)$  coincides with the
strict transform $f_\circ(W)$. Indeed, $f_*(W)$ and $f_\circ(W)$ coincide on the open subset of $Y$ on which 
$f^{-1}$ is a local isomorphism, and this open subset has codimension $\geq 2$. 

Recall that the N\'eron-Severi group $\NS(X)$ is the free abelian group of codimension $1$ cycles modulo
cycles which are numerically equivalent to $0$. Its rank is finite and is called the Picard number of $X$. 

\begin{thm}\label{thm:pseudo-automorphisms-neron-severi}
The action of pseudo-isomorphisms on N\'eron-Severi groups is functorial: $(g\circ f)_*=g_*\circ f_*$ for all 
pairs of pseudo-isomorphisms $f\colon X\dasharrow Y$ and $g\colon Y\dasharrow Z$.
If $X$ is a normal projective variety, the group $\Psaut(X)$ acts linearly on the N\'eron-Severi group $\NS(X)$; 
this provides a morphism 
\[
\Psaut(X)\to \GL(\NS(X)).
\]
The kernel of this morphism is contained in $\Aut(X)$ and contains $\Aut(X)^0$ as a finite index subgroup. 
\end{thm}

As a consequence, if $X$ is projective the group $\Psaut(X)$ is an extension of a discrete linear subgroup of 
$\GL(\NS(X))$ by an algebraic group. 

\begin{proof}
The first statement follows from the equality $f_*=f_\circ$ on  divisors. The second follows from the first. 

For the last assertion, we shall need the following fact: {\sl{if $f\colon X\dasharrow Y$ is a pseudo-isomorphism  
between normal projective varieties such that $f_*(H_X)=H_Y$
for some pair of very ample divisors $H_X$ and $H_Y$ on $X$ and $Y$, then, $f$ is an isomorphism}} (see \cite{KSC:book} exercise 5.6, and \cite{Matsusaka-Mumford:1964}). 
Indeed, $(f^{-1})_*=(f^{-1})_\circ$ maps the linear system $\vert H_Y\vert$ bijectively onto $\vert H_X\vert$; if $f$ had an indeterminacy point, 
there would be a curve in its graph whose first projection would be a point $q\in X$ and second projection  would be a curve $C \subset Y$: 
since all members of $\vert H_Y\vert$ intersect $C$,  all members of $\vert H_X\vert$ should contain $q$, in contradiction with the very ampleness of $H_X$.

We can now study the kernel $K$ of the  representation $\Psaut(X)\to \GL(\NS(X))$. Fix an embedding $X\subset \P^m_\bfk$
and denote by $H_X$ the polarization given by hyperplane sections. For every $f$ in $K$, $f_*(H_X)$ is very ample
because its class in $\NS(X)$ coincides with the class of $H_X$. Thus, by what has just been proven, $f^*$ is an automorphism. To conclude, note that
$\Aut(X)^0$ has finite index in the kernel of the action of $\Aut(X)$ on $\NS(X)$: see \cite{Matsusaka:1958}, Theorem~6 in \S 11, and its extension 
to arbitrary projective varieties in \cite{Grothendieck:Bourbaki-Hilbert}, page 268; and see \cite{Lieberman:1978}, 
Proposition 2.2, for compact k\"ahler manifold. 
\end{proof}

\subsection{Affine varieties}\label{par:affine}
The group $\Psaut(\A^n_\bfk)$ coincides with the group $\Aut(\A^n_\bfk)$ of polynomial 
automorphisms of the affine space $\A^n_\bfk$: this is a special case of the following proposition. 

\begin{pro}\label{pro:pseudo-automorphisms-affine}
Let $Z$ be an affine variety. 
If $Z$ is locally factorial, the group $\Psaut(Z)$ coincides with the group $\Aut(Z)$.
\end{pro}

\begin{proof}
Fix an embedding $Z\to \A^m_\bfk$. Rational functions on $Z$ are restrictions of rational functions 
on $\A^m_\bfk$. Thus, every birational transformation $f\colon Z\dasharrow Z$ is given by rational formulas
$
f  (x_1, \ldots, x_m)= (f_1, \ldots, f_m)
$
where each $f_i$ is a rational function. To show that $f$ is an automorphism, we only need to prove
that $f_i$ is in the local ring ${\mathcal{O}}_{Z,x}$ for every index $i$ and every point  $x\in Z$. Otherwise
\[
f_i=\frac{p_i}{q_i} \]
where $p_i$ and $q_i$ are relatively prime elements of the local ring ${\mathcal{O}}_{Z,x}$, and 
$q_i$ is not invertible. Fix  an irreducible factor $h$ of $q_i$, and 
 an open neighborhood $U$ of $x$ on which $p_i$, $q_i$ and $h$ are defined. 
The hypersurfaces $W_U(p_i)=\{z \in U; \; p_i(z)=0\}$
and $W_U(h)=\{z \in U; \; h(z)=0\}$ have no common components, hence the latter would be mapped to infinity by $f$, 
and $f$ would not be a pseudo-automorphism. This contradiction shows that all $f_i$ are regular and $f$ is an automorphism.
\end{proof}

\begin{eg} Consider the affine quadric cone $Q\subset \A^3$ defined by the equation $z^2=x^2+y^2$; the origin is a 
singular point of $Q$, and it is not factorial at that point, because the relation $z^2=(x+iy)(x-iy)$ shows that $z^2$
can be factorized in two distinct ways. Now, consider the affine variety $Z=Q\times \A^1\subset \A^4$, with coordinates
$(x,y,z,t)$. The map $f(x,y,z,t)=(x,y,z,t+ \frac{z}{x+iy})$ is a birational transformation of $Z$. 
The indeterminacy sets of $f$ and $f^{-1}$ coincide with 
the vertical line $ \{(0,0,0)\}\times \A^1$ and $f$ and $f^{-1}$ do not contract any hypersurface, hence $f$ is a pseudo-isomorphism. But $f$ is not an automorphism. 
\end{eg}

\section{Groups with Property {\FW}}\label{par:FW}

\subsection{Commensurated subsets and cardinal definite length functions (see \cite{Cornulier:Survey-FW})}

Let $G$ be a group, and $G \times S\to S$ an action of $G$ on a set $S$. Let
$A$ be a subset of~$S$. As in the Introduction, one says that $G$ {\bf{commensurates}} $A$ if the symmetric difference
$A\triangle gA$ is finite for every element $g\in G$. One says that $G$
{\bf{transfixes}} $A$ if there is a subset $B$ of $S$ such that $A\triangle B$ is finite and $B$ is $G$-invariant:
$gB=B$ for every $g$ in $G$. If $A$ is transfixed, then it is commensurated.
Actually, $A$ is transfixed if and only if the function $g\mapsto\#(A\triangle gA)$ is bounded on $G$. 

A group $G$ has {\bf{Property {\FW} }} if, given any action of $G$ on a set $S$, all
commensurated subsets of $S$ are automatically transfixed. 
More generally, if $H$ is a subgroup of $G$, then $(G,H)$ has {\bf{relative Property {\FW}}} if every commensurating 
action of $G$ is transfixing in restriction to $H$. This means that, if $G$  acts
on a set $S$ and commensurates a subset $A$, then $H$ transfixes automatically $A$. The case $H=G$ is Property {\FW} for $G$.

We refer to \cite{Cornulier:Survey-FW} for a detailed study of Property {\FW}. The next paragraphs present the two main 
sources of examples for groups with Property {\FW} or its relative version, namely Property~(T) 
and distorted subgroups.  

\begin{rem}
Property {\FW} should be thought of 
as a rigidity property.  To illustrate this idea, consider a group $K$ with Property {\PW}; by definition, this means 
that $K$ admits a commensurating action on a set $S$, with a commensurating subset $C$ such that the 
function $g\mapsto\#(C\triangle gC)$   has finite fibers. If $G$ is a group 
with Property {\FW}, then, every homomorphism $G\to K$ has  finite image.
\end{rem}

\subsection{Property {\FW} and Property (T) (see \cite{Cornulier:Survey-FW})}\label{par:(FW)-(T)}

One can rephrase Property {\FW} as follows: $G$ has Property {\FW} 
if and only if every isometric action on an ``integral Hilbert space'' $\ell^2(X,\Z)$ 
has bounded orbits, for any discrete set $X$.

A group has Property {\FH} if all its isometric actions on Hilbert spaces have fixed points. 
More generally, a pair $(G,H)$ of a group $G$ and a subgroup $H\subset G$ has relative Property {\FH} 
if every isometric $G$-action on a Hilbert space has an $H$-fixed point. Thus, the relative Property {\FH} implies the
relative Property {\FW}. 

By a theorem of Delorme and Guichardet, Property {\FH} is equivalent to Kazhdan's Property (T)  for countable
groups; this is the viewpoint we used to describe Property (T) in the introduction (see \cite{delaHarpe-Valette:Ast} for other 
equivalent definitions). Thus, Property (T) implies Property {\FW}. 
Kazhdan's Property (T) is satisfied by  lattices in semisimple Lie groups all of whose simple factors have Property (T), 
for instance if all simple factors have real rank $\ge 2$. For example, $\SL_3(\Z)$ satisfies Property (T). 

Property {\FW} is actually conjectured to hold for all irreducible lattices in semi-simple Lie groups of real rank $\ge 2$, 
such as $\SL_2(\mathbf{R})^k$ for $k\ge 2$. (here, irreducible means that the projection of the lattice {\it modulo} 
every simple factor is dense.) This is known in the case of a semisimple Lie group admitting at least one noncompact 
simple factor with Kazhdan's Property (T), for instance in $\SO(2,3)\times\SO(1,4)$, which admits 
irreducible lattices (see \cite{Cornulier:MathZ}).

\subsection{Distortion}

Let $G$ be a group. An element $g$ of $G$ is {\bf distorted} in $G$ if there exists a finite subset $\Sigma$ of $G$ generating a subgroup $\langle\Sigma\rangle$ containing $g$, such that $\lim_{n\to\infty}\frac{1}{n}|g^n|_\Sigma=0$; here, $\vert g \vert_\Sigma$ is the length of $g$ 
with respect to the set $\Sigma$. If $G$ is finitely generated, this condition  holds for some $\Sigma$ if and only if it holds for every finite generating subset of $G$. For example, every finite order element is distorted. 

\begin{eg}
Let $K$ be a field. The distorted elements of  $\SL_n(K)$  are exactly the virtually unipotent elements, 
that is, those elements  whose eigenvalues are all roots of unity; in positive characteristic, these are elements of finite order. 
By results of Lubotzky, Mozes, and Raghunathan (see \cite{lubotzky1993cyclic,Lubotzky-Mozes-Raghunathan:2001}), 
the same characterization holds in  $\SL_n(\mathbf{Z})$ when $n\ge 3$;  it also holds in 
$\SL_n(\mathbf{Z}[\sqrt{d}])$ when  $n\geq 2$ and $d\ge 2$ is not a perfect square. 
In contrast, in $\SL_2(\mathbf{Z})$, every element of 
infinite order is undistorted.
\end{eg}

\begin{lem}[see \cite{Cornulier:Survey-FW}]
Let $G$ be a group, and $H$ a finitely generated abelian subgroup of $G$ consisting of distorted elements. Then, the pair $(G,H)$ has relative 
Property~{\FW}. 
\end{lem}

This lemma provides many examples. For instance, if $G$ is any finitely generated nilpotent group and $G'$ is 
its derived subgroup, then $(G,G')$ has relative Property {\FH}; this result is due to Houghton, in a more general 
formulation encompassing polycyclic groups (see \cite{Cornulier:Survey-FW}).
Bounded generation by distorted unipotent elements can also be used to obtain nontrivial examples of groups with 
Property~{\FW}, including the above examples  $\SL_{n}(\mathbf{Z})$ for $n\geq 3$, and $\SL_{n}(\mathbf{Z}[\sqrt{d}])$. 
The case of $\SL_{2}(\mathbf{Z}[\sqrt{d}])$ is particularly interesting because it does not have Property (T).

\subsection{Subgroups of $\PGL_2(\bfk)$ with Property {\FW}}\label{par:Bass}

If a group $G$ acts on a tree $T$ by graph automorphisms, then $G$ acts on the set $E$ of 
directed edges of $T$ ($T$ is non-oriented, so each edge gives rise to a pair of opposite directed edges). 
Let $E_v$ be the set of directed edges pointing towards a vertex $v$. Then $E_v\triangle E_w$ is the set 
of directed edges lying in the segment between $v$ and $w$; it is finite of cardinality $2d(v,w)$, 
where $d$ is the graph distance. The group $G$ commensurates  $E_v$ for every $v$, and 
$\#(E_v\triangle gE_v)=2d(v,gv)$. Consequently, if $G$ has Property {\FW}, then it has Property (FA) 
meaning that every action of $G$ on a tree has bounded orbits. 
Combined with Proposition~5.B.1 of \cite{Cornulier:Survey-FW}, this argument  leads to the following lemma.

\begin{lem}[See \cite{Cornulier:Survey-FW}]
Let $G$ be a group with Property {\FW}, then all finite index subgroups of $G$ have Property {\FW}, and hence have Property~(FA). 
Conversely, if a finite index subgroup of $G$ has Property~{\FW}, then so does $G$.
\end{lem}
 
On the other hand, Property~(FA) is not stable by taking finite index subgroups.

\begin{lem}\label{lem:pgl2-fw}
Let $\bfk$ be an algebraically closed field and $\Lambda$ be a subgroup of $\GL_2(\bfk)$. 
\begin{enumerate}
\item  $\Lambda$ has a finite orbit on the projective line if and only if it is virtually solvable, if and only if its Zariski closure 
does not contain $\SL_2$.

\item Assume that all finite index subgroups of $\Lambda$ have Property {(FA)} (e.g., $\Lambda$ has Property FW). 
If the action of $\Lambda$ on the projective line preserves a non-empty, finite set, then $\Lambda$ is finite.
\end{enumerate}
\end{lem}
The proof of the first assertion is standard and omitted. The second assertion follows directly from the first one. 

In what follows, we denote by $\oZ\subset \overline{\Q}$ the ring of algebraic integers (in some fixed algebraic closure $\overline{\Q}$ of $\Q$).
 
\begin{thm}[Bass \cite{Bass:Pacific}]\label{thm:Bass}
Let $\bfk$ be an algebraically closed field.
\begin{enumerate}
\item If $\bfk$ has positive characteristic, then $\GL_2(\bfk)$ has no infinite subgroup with Property~(FA).
\item Suppose that $\bfk$ has characteristic zero and that $\Gamma\subset\GL_2(\bfk)$ is a countable subgroup with Property (FA), and is not virtually abelian. Then $\Gamma$ acts irreducibly on $\bfk^2$, and is conjugate to a subgroup of $\GL_2(\oZ)$. If moreover $\Gamma\subset\GL_2(K)$ for some subfield $K\subset\bfk$ containing $\overline{\Q}$, then we can choose the conjugating matrix to belong to $\GL_2(K)$.
\end{enumerate}
\end{thm} 
\begin{proof}[On the proof]

According to  \cite[\S 6, Th.~15]{Serre:AASL2}, a countable group with Property (FA) is finitely generated. 
Thus, if $\Gamma\subset \GL_2(\bfk)$ has Property (FA) it is contained in $\GL_2(K)$ for some finitely generated field $K\subset \bfk$ (choose $K$ to be the field generated by entries of a finite generating subset of $\Gamma$). 
Then, the first statement follows from Corollary 6.6 of \cite{Bass:Pacific}. 

Now, assume that the characteristic of $\bfk$ is $0$. Since a group with Property (FA) has no infinite cyclic quotient,
and is not a non-trivial amalgam, Theorem~6.5 of \cite{Bass:Pacific} can  be applied, giving the first assertion of (2) (see also the first Theorem in \cite{Bass:Smith}).
For the last assertion, we have $\Gamma\cup B\Gamma B^{-1}\subset \GL_2(K)$ for some  $B\in\GL_2(\bfk)$ such that $B\Gamma B^{-1}\subset
\GL_2(\oZ)$; we claim that this implies that $B\in \bfk^*\GL_2(K)$. First, since $\Gamma$ is absolutely irreducible, this implies that $B\mathcal{M}_2(K)B^{-1}\subset \mathcal{M}_2(K)$. The conclusion follows from Lemma \ref{algl} below, which can be of independent interest.
\end{proof} 

\begin{lem}\label{algl}
Let $K\subset L$ be fields. Then the normalizer $\{B\in\GL_2(L):B\mathcal{M}_2(K)B^{-1}\subset \mathcal{M}_2(K)\}$ is reduced to $L^*\GL_2(K)=\{\lambda A:\lambda\in L^*,A\in\GL_2(K)\}$.
\end{lem}
\begin{proof}
Write 
\[
B=\begin{pmatrix}b_1 & b_2\\ b_3 & b_4\end{pmatrix}.
\] 
Since $BAB^{-1}\in \mathcal{M}_2(K)$ for the three elementary matrices $A\in\{E_{11},E_{12},E_{21}\}$, we deduce by a plain computation that $b_ib_j/b_kb_\ell\in K$ for all $1\le i$, $j$, $k$, $\ell\le 4$ such that $b_kb_\ell\neq 0$. In particular, for all indices $i$ and $j$ such that $b_i$ and $b_j$ are nonzero, the quotient $b_i/b_j=b_ib_j/b_j^2$ belongs to $K$. It follows that $B\in L^*\GL_2(K)$.
\end{proof} 
 
 \begin{cor}\label{coro:Bass-k(C)}
 Let $\bfk$ be an algebraically closed field. Let $C$ be a projective curve over $\bfk$, and let $\bfk(C)$ be
 the field of rational functions on the curve $C$. Let $\Gamma$ be an infinite subgroup of $\PGL_2(\bfk(C))$.
 If $\Gamma$ has Property (FA), then
 \begin{enumerate}
 \item the field $\bfk$ has characteristic $0$; 
 \item there is an element of $\PGL_2(\bfk(C))$ that conjugates $\Gamma$ 
 to a subgroup of $\PGL_2(\oZ)\subset \PGL_2(\bfk(C))$.\qed
 \end{enumerate}
 \end{cor}

\section{Divisorial valuations, hypersurfaces, and the action of $\Bir(X)$}\label{par:div-hyp-action}

The group of birational transformations $\Bir(X)$ acts on the function field $\bfk(X)$, hence
also on the set of valuations of $\bfk(X)$. The subset of divisorial 
valuations is invariant, and the centers of those valuations correspond to irreducible
hypersurfaces in various models of $X$.  In this way, we obtain a natural action of $\Bir(X)$ on (reduced, irreducible) hypersurfaces
in all models of $X$; this section presents this classical construction (we refer to \cite{ZS}, chapter VI, and \cite{Vaquie}
for detailed references). 

\subsection{Divisorial valuations}
Consider a  projective variety $X$ over an algebraically closed field $\bfk$ and let $\bfk(X)$ 
be its function field. A discrete, rank $1$, valuation $v$ on $\bfk(X)$ is a function on the multiplicative
group $\bfk(X)^*$ with values in the cyclic group $\Z$ such that 
\begin{itemize}
\item[(i)] $v(\varphi \psi)=v(\varphi)+v(\psi)$ and
$v(\varphi+\psi)\geq \min(v(\varphi), v(\psi))$, $\forall \varphi, \; \psi \in \bfk(X)^*$, 
\item[(ii)] $v$ vanishes on the set of constant functions $\bfk\subset \bfk(X)$, 
\item[(iii)]  $v(\bfk(X))=\Z$ (we assume that the value group is equal to $\Z$ in this article).
\end{itemize}
Its valuation ring is the subring $R_v\subset \bfk(X)$ defined by $R_v=v^{-1}(\Z_+)$, 
where $\Z_+$ is the set of non-negative integers. This ring contains a unique maximal 
ideal, namely  $\m_v=v^{-1}(\Z_+^*)$, where $\Z_+^*$ is the set of positive integers. 
The residue field is the quotient field $\bfk(X)_v= R_v/\m_v$; 
if its transcendence degree is equal to $\dim(X)-1$, then $v$ is said to be a {\bf{divisorial 
valuation}} (see~\cite{ZS}, \S VI.14, \cite{Vaquie}, \S 10). We shall denote by $\dv(X)$ the set of divisorial valuations on $\bfk(X)$.

Any birational map $f\colon X\dasharrow X'$ determines an isomorphism of function 
fields and transports divisorial valuations to divisorial valuations:  if 
 $v$ is a divisorial valuation on $\bfk(X)$, then $f(v)(\varphi):=v(\varphi\circ f)$ defines a divisorial 
 valuation on $\bfk(X')$. Indeed, the group of values is not modified by this action, and the residue fields $\bfk(X)_v$ 
and $\bfk(X)_{f(v)}$ are isomorphic. In this way, $\Bir(X)$ acts on $\dv(X)$.

\subsection{Hypersurfaces}

We now work with normal and projective varieties; we shall use that their singular loci, and
the indeterminacy loci of birational maps have codimension $\geq 2$ (in particular, the strict
transform of any hypersurface is well defined). 

Let $\pi\colon Y\to X$ be a birational morphism  between normal projective varieties. 
Let $E$ be a reduced, irreducible, hypersurface in $Y$. Since $Y$ is normal, it is
smooth at the generic point of $E$; thus, if $\varphi$ is an element of $\bfk(X)^*$, we can 
define the order of vanishing $v_E(\varphi)$ of $\varphi$ along $E$: $v_E(\varphi)=a\geq 0$ if $\varphi\circ \pi$ vanishes
at order $a$ along $E$, and $v_E(\varphi)=-a$ if $\varphi\circ \pi$ has a pole of order $a$
along $E$. Then, $v_E$ is a divisorial valuation, with residue field isomorphic to
$\bfk(E)$. One says that $v_E$ is the {\bf{geometric valuation}} associated to $E$ (or more precisely to 
$(\pi, E)$).
A theorem of Zariski asserts that {\sl{every divisorial valuation is  geometric}} (see~\cite{ZS}, \S VI.14, or \cite{Vaquie}, \S 10). 
Thus, one may define 
the set $\Hyp(X)$ of {\bf{irreducible hypersurfaces}} in all (normal) models of $X$ as the set of divisorial 
valuations $\dv(X)$. Any reduced and irreducible hypersurface $E$ in any model $Y\to X$ determines
such a point $E\in \Hyp(X)$; two divisors $E$ and $E'$ in two models $\pi\colon Y\to X$
and $\pi'\colon Y'\to X$ correspond to the same point in $\Hyp(X)$ if and only if the two 
valuations $v_E$ and $v_{E'}$ coincide, if and only if $E$ is the strict transform of $E'$ by the birational map $\pi^{-1}\circ \pi'\colon Y'\dasharrow Y$. 
The action of $\Bir(X)$ on valuations becomes an action by permutations on  $\Hyp(X)$, which we
denote by 
\begin{equation}
f_\bullet\colon E\in \Hyp(X)\mapsto f_\bullet (E); 
\end{equation}
it satisfies $f(v_{f_\bullet(E)})=v_E$. If $E$ is a reduced and irreducible hypersurface in the model
$\pi\colon Y\to X$, there is a birational morphism $\pi'\colon Y'\to X$ such that 
$\pi'^{-1}\circ f\circ \pi$ does not contract $E$; then, the strict transform of $E$ by $\pi'^{-1}\circ f\circ \pi$
is a reduced, irreducible hypersurface $E'$ in $Y'$ that represents the 
point $f_\bullet(E)$ in $\Hyp(X)$. 

More generally, if $f\colon X\dasharrow X'$ is a birational map between normal projective varieties, 
we obtain a bijection $f_\bullet\colon \Hyp(X)\to \Hyp(X')$.

\subsection{The subset $\Hy(X)$}

Let $\Hy(X)\subset \Hyp(X)$ be the subset of all reduced, irreducible hypersurfaces of the normal variety $X$. 
Recall that a hypersurface is contracted by a birational map
if its strict transform is a subset of codimension $>1$.  Given a birational map $f:X\dasharrow X'$ 
between normal  projective varieties, define 
\[
\exc(f)=\# \left\{ S\in \Hy(X);\;\; f\; {\text{ contracts  }}\; S \right\}.
\]
This is the {\bf{number of contracted hypersurfaces}}  
$S\in \Hy(X)$ by $f$. In the following proposition,  $f_\circ$ denotes the strict transform
and  $f_\bullet$ the action on $\Hyp(X)$.  

\begin{pro}\label{pro:contraction-hyp-estimate}
Let $f:X\dasharrow X'$ be a birational transformation between normal irreducible projective varieties. 
Let $S$ be an element of $\Hy(X)$.
\begin{enumerate}
\item\label{ipascontra} If $S\in  (f^{-1})_\circ\Hy(X')$,  then $f_\bullet(S)=f_\circ(S)\in\Hy(X')$.
\item\label{icontra} If $S\notin (f^{-1})_\circ \Hy(X')$,  then $f_\circ(S)$ has codimension $\ge 2$ (i.e. $v$ contracts $S$),
and $f_\bullet(S)$ is an element of $\Hyp(X')\smallsetminus \Hy(X')$. 
\item\label{iconclusion} The symmetric difference $f_\bullet(\Hy(X))\triangle \Hy(X')$ contains $\exc(f)+\exc(f^{-1})$ elements.
\end{enumerate}
\end{pro}

\begin{proof}
Let $U$ be the  complement of $\Ind(f)$ in $X'$. Since, by Theorem \ref{normal2}, $\Ind(f)$ has codimension $\ge 2$, no 
$S\in\Hy(X)$ is contained in $\Ind(f)$.
Let us prove (\ref{ipascontra}). This is clear when $f$ is  a birational morphism. 
To deal with the general case, write $f=g\circ h^{-1}$ where  
$h\colon Y\to X$ and $g\colon Y\to X'$ are birational morphisms from a normal variety $Y$.
Since $h$ is a birational morphism, $h^\bullet(S)=h^\circ(S)\subset\Hy(Y)$; since $S$ is not contracted by $f$, 
$g_\bullet (h^\circ(S)) = g_\circ(h^\circ(S)) \in \Hy(X')$. 
Thus,  $f_\bullet(S)=g_\bullet (h^\bullet(S))$ coincides with the strict transform $f_\circ(S)\in\Hy(X')$. 

Now let us prove (\ref{icontra}), assuming thus that $S\notin (f^{-1})_\circ\Hy(X')$. 
Let $S''\in\Hy(Y)$ be the hypersurface $(h^{-1})_\bullet(S)=(h^{-1})_\circ(S)$. Then $h(S'')=S$. If $g_\circ (S'')$ is a hypersurface $S'$, 
then $(f^{-1})_\circ (S')=S$, contradicting $S\notin (f^{-1})_\circ\Hy(X')$. 
Thus, $g$ contracts $S''$ onto a subset $S'\subset X'$ of codimension $\geq 2$. 
Since $S'=f_\circ(S)$, assertion  (\ref{icontra}) is proved. 

Assertion~(\ref{iconclusion}) follows from the previous two assertions. 
\end{proof}

\begin{eg}
Let $g$ be a birational transformation of $\P^n_\bfk$ of degree $d$, meaning that $g$ is defined by $n+1$ homogeneous polynomials of degree
$d$ without common factor of positive degree, or equivalently that $g^*(H)\simeq dH$ where $H$ is any
hyperplane of $\P^n_\bfk$. The exceptional set of $g$ has degree $(n+1)(d-1)$; thus, 
$\exc_{\P^n_\bfk}(g)\leq (n+1)(d-1)$. More generally, if $H$ is a polarization of $X$, then 
$\exc_X(g)$ is bounded from above by a function that depends only on  the degree $\deg_H(g):=(g^*H)\cdot H^{\dim(X)-1}$.
\end{eg} 
  
\begin{thm}\label{thm:def-action-hypX}
Let $X$ be a normal projective variety. 
The group $\Bir(X)$ acts faithfully by permutations on the set $\Hyp(X)$ via the homomorphism $g\mapsto g_\bullet$
from $\Bir(X)$ to $\Bij(\Hyp(X))$.
This action commensurates the subset $\Hy(X)$ of $\Hyp(X)$: for every $g\in\Bir(X)$, 
$
\vert g_\bullet(\Hy(X))\triangle \Hy(X)\vert = \exc(g)+\exc(g^{-1}).
$
\end{thm}

It remains only to prove that the homomorphism 
$f\in \Bir(X)\mapsto f_\bullet \in \Bij(\Hyp(X))$ is injective. An element of its kernel satisfies $f_\circ (W)=W$ for every 
hypersurface $W$ of $X$. Embedding $X$ in some projective 
space~$\P^m_\bfk$, every point of $X(\bfk)$ is the intersection 
of finitely many irreducible hyperplane sections of $X$: since all these sections are fixed 
by $f$, every point is fixed by $f$, and $f$ is the identity. 
  
 \subsection{Products of varieties}

Let $X$ and $Y$ be irreducible, normal projective varieties. 
Consider the embedding of $\Bir(X)$ into $\Bir(X\times Y)$ given by 
the action $f\cdot(x,y)=(f(x),y)$ for $f\in\Bir(X)$. The injection 
$j_Y$ of $\Hy(X)$ into $\Hy(X\times Y)$ given by $j_Y(S)=S\times Y$ 
 extends to an injection of $\Hyp(X)$ into $\Hyp(X\times Y)$; this inclusion is $\Bir(X)$-equivariant. 
The following result will be applied to  Corollary~\ref{nondistab}.

\begin{pro}\label{pro:produ}
Let a group $\Gamma$ act on $X$ by birational transformations. Then $\Gamma$ transfixes $\Hy(X)$ in $\Hyp(X)$ if and only if it transfixes $\Hy(X\times Y)$ in $\Hyp(X\times Y)$. More precisely, the subset $\Hy(X\times Y)\smallsetminus j_Y(\Hy(X))$ is $\Bir(X)$-invariant.
\end{pro}
\begin{proof}
The reverse implication is immediate.
The direct one follows from the latter statement, which we now prove. The projection of a hypersurface 
$S\in \Hy(X\times Y)\smallsetminus j_Y(\Hy(X))$ on  $X$ is surjective. For $f\in\Bir(X)$, $f$ 
induces an isomorphism between dense open subsets $U$ and $V$ of $X$, and hence between $U\times Y$ 
and $V\times Y$; in particular, $f$ does not contract $S$. 
This shows that $f$ stabilizes $\Hy(X\times Y)\smallsetminus j_Y(\Hy(X))$.
\end{proof}

\section{Pseudo-regularization of birational transformations}

In this section, the action of $\Bir(X)$ on $\Hyp(X)$ is used to characterize and
study groups of birational transformations that are pseudo-regularizable, in the sense of Definition~\ref{d_psr}.
As before, $\bfk$ is an algebraically closed field. 

\subsection{An example}\label{par:example-intro6}

Consider the birational transformation $f(x,y)=(x+1,xy)$ of $\P^1_\bfk\times \P^1_\bfk$. The vertical 
curves $C_i=\{x=-i\}$, $i\in \Z_+$, are exceptional curves for the cyclic group $\Gamma=\langle f \rangle$: each 
of these curves is contracted by an element of $\Gamma$ onto a point, namely $f^{i+1}_\circ(C_i)=(1,0)$. 
Let $\varphi\colon Y\dasharrow \P^1_\bfk\times \P^1_\bfk$ be a birational map, and let $\U$ be a non-empty 
open subset of $Y$. Consider the subgroup $\Gamma_Y:=\varphi^{-1}\circ \Gamma\circ \varphi$ of $\Bir(Y)$. 
If $i$ is large enough, $\varphi^{-1}_\circ(C_i)$ is an irreducible curve $C'_i\subset Y$, and these curves $C'_i$ 
are pairwise distinct, so that most of them intersect $\U$. For positive integers $m$, $f^{i+m}$ maps 
$C_i$ onto $(m,0)$, and $(m,0)$ is not an indeterminacy point of $\varphi^{-1}$ if $m$ is large. Thus, 
$\varphi^{-1}\circ f^m\circ \varphi$ contracts $C'_i$, and $\varphi^{-1}\circ f^m\circ \varphi$ is not a pseudo-automorphism
of $\U$. This argument proves the following lemma. 

\begin{lem}\label{lem:counter-example}
Let $X$ be the surface $\P^1_\bfk\times \P^1_\bfk$.
Let $f\colon X\dasharrow X$ be defined by $f(x,y)=(x+1,xy)$, and let $\Gamma$ be the 
subgroup generated by $f^\ell$, for some $\ell \geq 1$. Then the cyclic group $\Gamma$ is not pseudo-regularizable.
\end{lem}

This shows that Theorem~A requires an assumption on $\Gamma$. 
More generally, a subgroup $\Gamma\subset \Bir(X)$ cannot be pseudo-regularized if
\begin{enumerate}
\item[(a)] $\Gamma$ contracts a family of hypersurfaces $W_i\subset X$ whose union is Zariski dense
\item[(b)] the union of all strict transforms $f_\circ(W_i)$, for $f\in \Gamma$ contracting $W_i$, is a subset of $X$ whose
Zariski closure has codimension at most $1$. 
\end{enumerate}

\subsection{Characterization of pseudo-isomorphisms}
Recall that $f_\bullet$ denotes the bijection $\Hyp(X)\to \Hyp(X')$ 
which is induced by a birational map $f\colon X\dasharrow X'$. Also, for any nonempty open subset $U\subset X$, we define $\Hy(U)=\{H\in\Hy(X):H\cap U\neq\emptyset\}$; its complement in $\Hy(X)$ is finite.

\begin{pro}\label{pro:charact-pseudo-isom} Let $f:X\dasharrow X'$ be a birational map between normal projective varieties. 
Let $U\subset X$ and $U'\subset X'$ be two dense open subsets. Then, $f$ induces a pseudo-isomorphism $U\dasharrow U'$
if and only if $f_\bullet(\Hy(U))=\Hy(U')$. 
\end{pro}

\begin{proof}
If $f$ restricts to a pseudo-isomorphism $U\dasharrow U'$, then $f$ maps every hypersurface of $U$ to a
hypersurface of $U'$ by strict transform. And $(f^{-1})_\circ$ is an inverse for $f_\circ \colon \Hy(U)\to \Hy(U')$.
Thus, $f_\bullet(\Hy(U))=f_\circ(\Hy(U)=\Hy(U')$. 

Now, assume that $f_\bullet(\Hy(U))=\Hy(U')$.  Since $X$ and $X'$ are normal, 
$\Ind(f)$ and $\Ind(f^{-1})$ have codimension $\ge 2$ (Theorem \ref{normal2}). 

Let $f_{U,U'}$ be the birational map from $U$ to $U'$ which is induced by $f$. 
The indeterminacy set of $f_{U,U'}$ is contained in the union of the set $\Ind(f)\cap U$ and the set of points $x\in U\smallsetminus \Ind(f)$
which are mapped by $f$ in the complement of $U'$; this second part of $\Ind(f_{U,U'})$ has codimension $2$, because otherwise
there would be an irreducible hypersurface $W$ in $U$ which would be mapped in $X'\smallsetminus U'$, contradicting the equality
$f_\bullet (\Hy(U))=\Hy(U')$. Thus, the indeterminacy set of $f_{U,U'}$ has codimension $\geq 2$. 
Changing $f$ in its inverse $f^{-1}$, we see that the indeterminacy set of $f^{-1}_{U',U}\colon U'\dasharrow U'$ has codimension $\geq 2$ too. 

If $f_{U,U'}$ contracted an irreducible hypersurface $W\subset U$ onto a subset of $U'$ of
codimension $\geq 2$, then $f_\bullet(W)$ would not be contained in $\Hy(U')$ (it would correspond to an element of 
$\Hyp(X')\smallsetminus \Hy(X')$ by Proposition~\ref{pro:contraction-hyp-estimate}). 
Thus, $f_{U,U'}$ satisfies the first property of Proposition~\ref{pro:pseudo-isomorphism} 
and, therefore, is a pseudo-isomorphism. 
\end{proof}

\subsection{Characterization of pseudo-regularization}

Let $X$ be a (reduced, irreducible) normal projective variety. Let $\Gamma$ be a subgroup 
of $\Bir(X)$. Assume that the action of $\Gamma$ on $\Hyp(X)$ fixes
(globally) a subset $A\subset \Hyp(X)$ such that 
\[
\vert A\triangle \Hy(X)\vert < +\infty.
\]
In other words, $A$ is obtained from $\Hy(X)$ by removing 
finitely many hypersurfaces $W_i \in \Hy(X)$ and adding finitely many hypersurfaces 
$W'_j\in \Hyp(X)\smallsetminus\Hy(X)$. Each~$W'_j$ comes
from an irreducible hypersurface in some model $\pi_j\colon X_j\to X$, and there 
is a model $\pi\colon Y\to X$ that covers all of them (i.e.\ $\pi\circ \pi_j^{-1}$ is a  morphism
from~$Y$ to $X_j$ for every $j$). Then, $\pi^\circ(A)$ is a subset of $\Hy(Y)$. 
Changing $X$ into $Y$, $A$ into~$\pi^\circ(A)$, and $\Gamma$ into $\pi^{-1}\circ \Gamma\circ \pi$, we may assume
that 
\begin{enumerate}
\item $A=\Hy(X)\smallsetminus\{E_1, \ldots, E_\ell\}$ where the $E_i$ 
are $\ell$ distinct irreducible hypersurfaces of $X$, 

\item the action of $\Gamma$ on $\Hyp(X)$ fixes the set $A$. 
\end{enumerate}
In what follows, we denote by $\U$ the Zariski open subset $X\smallsetminus \cup_i E_i$
and by $\partial X$ the set $X\smallsetminus \U= E_1\cup \cdots \cup E_\ell$, considered as the boundary 
of the compactification $X$ of~$\U$.

\begin{lem}\label{lem:pseudo-on-U}
The group $\Gamma$ acts by pseudo-automorphisms on the open subset $\U$. If $\U$ is smooth (or locally factorial) and there is an  ample 
divisor $D$ whose support coincides with $\partial X$, then $\Gamma$ acts by automorphisms on $\U$.
\end{lem}

In this statement, we say that the support of a divisor $D$ coincides with $\partial X$ if 
$D=\sum_{i} a_i E_i$ with $a_i>0$ for every $1\leq i\leq \ell$.

\begin{proof}
Since $A=\Hy(\U)$ is $\Gamma$-invariant, Proposition~\ref{pro:charact-pseudo-isom} shows that $\Gamma$ acts by pseudo-automorphisms on $\U$.
Since $D$ is an ample divisor, some positive multiple $mD$ is very ample, and the complete
linear system $\vert mD\vert$ provides an embedding of $X$ in a projective space. The divisor
$mD$ corresponds to a hyperplane section of $X$ in this embedding, and the open subset $\U$ 
is an affine variety because the support of $D$ is equal to $\partial X$. 
Proposition~\ref{pro:pseudo-automorphisms-affine} concludes the proof of the lemma. 
\end{proof}

By Theorem~\ref{thm:def-action-hypX}, every subgroup of $\Bir(X)$ acts on $\Hyp(X)$ and 
commensurates $\Hy(X)$. If $\Gamma$ transfixes $\Hy(X)$, 
there is an invariant subset $A$ of $\Hyp(X)$ for which $A\triangle \Hy(X)$ is finite. 
Thus, one gets the following characterization of pseudo-regularizability (the converse being immediate).

\begin{thm}\label{thm:FW-pseudo-regularization}
Let $X$ be a normal projective variety over an algebraically closed field $\bfk$. 
Let $\Gamma$ be a subgroup of $\Bir(X)$. Then $\Gamma$ transfixes the subset $\Hy(X)$ of 
$\Hyp(X)$ if and only if $\Gamma$ is pseudo-regularizable. More precisely, if $\Gamma$ transfixes $\Hy(X)$, then there is a birational
morphism $\pi\colon Y\to X$ and a dense open subset $\U\subset Y$ such that $\pi^{-1}\circ \Gamma\circ \pi$
acts by pseudo-automorphisms on $\U$.
\end{thm}

Of course, this theorem applies directly when $\Gamma\subset \Bir(X)$ 
has property {\FW} because Theorem~\ref{thm:def-action-hypX} shows that $\Gamma$ commensurates 
$\Hy(X)$. This proves Theorem~A.

\begin{rem}
Assuming ${\mathrm{char}}(\bfk)=0$, we may apply the resolution of singularities and work in the category of smooth varieties.
As explained in Remark~\ref{rem:intro-extreme-cases} and Lemma~\ref{lem:pseudo-on-U}, there are two extreme cases, corresponding to 
an empty or an ample boundary $B=\cup_i E_i$. 
 If $\U=Y$, $\Gamma$ acts by pseudo-automorphisms on the projective model~$Y$. As explained in 
Theorem~\ref{thm:pseudo-automorphisms-neron-severi}, $\Psaut(Y)$ is an extension of a subgroup of $\GL(\NS(Y))$ by an algebraic group which contains 
$\Aut(Y)^0$ as a finite index subgroup.
 If $\U$ is affine, $\Gamma$ acts by automorphisms on $\U$. The group $\Aut(\U)$ may be huge 
($\U$ could be the affine space), but there are techniques to study groups of automorphisms 
that are not available for birational transformations; for instance $\Gamma$ is residually finite 
and virtually torsion free if $\Gamma$ is a group of automorphisms generated by finitely many elements
(see \cite{Bass-Lubotzky:1983}). 
\end{rem}

\subsection{Distorted elements}\label{par:distorsion}
 Theorem~\ref{thm:FW-pseudo-regularization} may be applied when $\Gamma$ has Property {\FW}, or 
 for pairs $(\Lambda,\Gamma)$ with relative 
Property {\FW}. Here is one application:

\begin{cor}
Let $X$ be an irreducible projective variety. Let
$\Gamma$ be a distorted cyclic subgroup of $\Bir(X)$. Then $\Gamma$ is pseudo-regularizable.
\end{cor} 

The contraposition is useful to show that some elements of $\Bir(X)$ are undistorted. Let us state it in a strong ``stable'' way.

\begin{cor}\label{nondistab}
Let $X$ be a normal irreducible projective variety and let $f$ be an element of $\Bir(X)$ such that the cyclic group $\langle f\rangle$ does not transfix $\Hy(X)$ (i.e., $f$ is not pseudo-regularizable). Then  $\langle f\rangle$ is undistorted in $\Bir(X)$; more generally the cyclic subgroup $\langle f\times\mathrm{Id}_Y\rangle$ is undistorted in $\Bir(X\times Y)$ for every irreducible projective variety $Y$.
\end{cor}

The latter consequence indeed follows from Proposition \ref{pro:produ}. This can be applied to various examples, such as those in Example \ref{gatransfix}.

\section{Illustrating results}

\subsection{Surfaces whose birational group is transfixing}\label{surf_birt}

If $X$ is a projective curve, $\Bir(X)$ always transfixes $\Hy(X)$, since it acts by automorphisms on a smooth model of $X$. 
We now consider the same problem for surfaces.

\begin{pro}\label{ruled}
Let $X$ be a normal irreducible variety of positive dimension over an algebraically closed field $\bfk$. 
Then $\Bir(X\times\mathbb{P}^1)$ does not transfix $\Hy(X\times\mathbb{P}^1)$. 
\end{pro}
\begin{proof}
We can suppose that $X$ is affine and work in the model $X\times\mathbb{A}^1$. For $\varphi$ a nonzero regular function on $X$, define a regular self-map $f$ of $X\times\mathbb{A}^1$ by $f(x,t)=(x,\varphi(x)t)$. Denoting by $Z(\varphi)$ the zero set of $\varphi$, we remark that $f$ induces an automorphism of the open subset $(X\smallsetminus Z(\varphi))\times\mathbb{A}^1$. In particular, it induces a permutation of $\Hy((X\smallsetminus Z(\varphi))\times\mathbb{A}^1)$.  Set $M=\Hy(X\times\mathbb{A}^1)$. Since $f$ contracts the complement $Z(\varphi)\times\mathbb{A}^1$ to the subset $Z(\varphi)\times\{0\}$, which has codimension $\ge 2$, its action on $\Hyp(X\times\mathbb{A}^1)$ maps  the  codimension $1$ components of $Z(\varphi)\times\mathbb{A}^1$ outside $M$. Therefore $M\smallsetminus f^{-1}(M)$ is the set of irreducible components of $Z(\varphi)\times\mathbb{A}^1$. Its cardinal is equal to the number of irreducible components of $Z(\varphi)$. When $\varphi$ varies, this number is unbounded;  hence, $\Bir(X\times\mathbb{A}^1)$ does not transfix $\Hy(X\times\mathbb{A}^1)$.
\end{proof}

Varieties which are birational to the product of a variety and the projective line are said to be {\bf ruled}. 
Proposition~\ref{ruled} states that $\Bir(Y)$ does not transfix $\Hy(Y)$ when $Y$ is  ruled and of dimension $\ge 2$.
The converse holds for surfaces:

\begin{thm}\label{birtran} Let $\bfk$ be an algebraically closed field. 
Let $X$ be an irreducible normal projective surface over $\bfk$. The following are equivalent:
\begin{enumerate}
\item\label{notransfix} $\Bir(X)$ does not transfix $\Hy(X)$;
\item\label{kod} the Kodaira dimension of $X$ is $-\infty$;
\item\label{kod2} $X$ is ruled;
\item\label{kod3} there is no projective surface $Y$ that is birationally equivalent to $X$ and satisfies $\Bir(Y)=\Aut(Y)$.
\end{enumerate} 
\end{thm}
\begin{proof}
The equivalence between (\ref{kod}) and (\ref{kod2}) is classical (see~\cite{BPVDVH} and~\cite{Badescu, Liedtke}). 
The group $\Aut(Y)$ fixes $\Hy(Y)\subset \Hyp(Y)$, hence (\ref{notransfix}) implies (\ref{kod3}). 
If the Kodaira dimension of $X$ is $\geq 0$, then $X$ has a unique minimal model $X_0$, 
and $\Bir(X_0)=\Aut(X_0)$. Thus, (\ref{kod3}) implies~(\ref{kod}).
Finally, Proposition \ref{ruled} shows that  (\ref{kod2}) implies (\ref{notransfix}).
\end{proof}
 
\begin{thm}\label{thm:fin-gen-transfix}
Let $X$ be an irreducible projective surface over an algebraically closed field $\bfk$. 
The following are equivalent:
\begin{enumerate}
\item\label{fgnt} some finitely generated subgroup of $\Bir(X)$ does not transfix $\Hy(X)$;
\item\label{cynt} some cyclic subgroup of $\Bir(X)$ does not transfix $\Hy(X)$;
\item\label{kod21}
\begin{itemize}
\item $\bfk$ has characteristic $0$, and $X$ is birationally equivalent to the product of the projective line with a curve of genus 0 or 1, or
\item $\bfk$ has positive characteristic, and $X$ is a rational surface.
\end{itemize}
\end{enumerate} 
\end{thm}

\begin{eg}\label{exp2facile}
Let $\bfk$ be an algebraically closed field that is not algebraic over a finite field. Let $t$ be an element of infinite order in the multiplicative group $\bfk^*$. Then the birational transformation $g$ of $\P^2_\bfk$ given, in affine coordinates, by $(x,y)\mapsto (tx+1,xy)$ does not transfix $\Hy(\P^2_\bfk)$. Indeed, it is easy to show that the hypersurface $C=\{x=0\}$  satisfies, for $n\in\Z$, $f_\bullet^n(C)\in\Hy(\P^2_\bfk)$ if and only if $n\le 0$.
\end{eg}

\begin{eg}\label{p2compli}
Example \ref{exp2facile} works under a small restriction on $\bfk$. Here is an example over an arbitrary algebraically closed field $\bfk$. 
Let $L$ and $L'$ be two lines in~$\P^2_\bfk$ intersecting transversally 
at a point $q$. Let $f$ be a birational transformation of $\P^2_\bfk$ that contracts $L'$ onto $q$ and fixes $L$. For instance, in affine coordinates, the monomial map $(x,y)\mapsto (x,xy)$ contracts the $y$-axis onto
the origin, and fixes the $x$-axis. Assume that there is an open neighborhood $\U$ of $q$ such that $f$ does
not contract any curve in $\U$ except the line $L'$. Let $C$ be an irreducible curve that intersects $L$ and $L'$ transversally at $q$. 
Then, for every $n\geq 1$, the strict transform $f^n_\circ(C)$ is an irreducible curve, and its order of tangency 
with $L$ goes to infinity with $n$. Thus, the degree of $f^n_\circ(C)$ goes to infinity, and the $f^n_\circ(C)$
form an infinite sequence in $\Hy(\P^2_\bfk)$.

Now, assume that $C$ is contracted by $f^{-1}$ onto a point $p$,   $p\notin \Ind(f)$, and  $p$ is fixed by $f^{-1}$. Then,  
for every $m\geq 1$, $f^{-m}_\bullet(C)$ is not in $\Hy(\P^2_\bfk)$. This shows that the orbit of $C$ under the action of $f_\bullet$
intersects $\Hy(\P^2_\bfk)$ and its complement $\Hyp(\P^2_\bfk)\smallsetminus \Hy(\P^2_\bfk)$ on the infinite sets $\{f^n_\circ(C)\, ; \, n\geq 1\}$ and  $\{f^{-m}_\bullet(C)\, ; \, m\geq 1\}$. In particular, $f$ does not transfix $\Hy(\P^2_\bfk)$. 

Since such maps exist over every algebraically closed field $\bfk$, this example shows that property (\ref{cynt}) of Theorem~\ref{thm:fin-gen-transfix}
is satisfied for every rational surface $X$. 
\end{eg}

\begin{proof}
Trivially (\ref{cynt}) implies (\ref{fgnt}). 
Suppose that (\ref{kod21}) holds and let us prove (\ref{cynt}). The case $X=\P^1\times\P^1$ is already covered by Lemma \ref{lem:counter-example}
in characteristic zero, and by the previous example in positive characteristic. The case $X=C\times\P^1$ in characteristic zero, where $C$ is an elliptic curve, is similar. To see it, fix a point $t_0\in C$ and a rational function $\varphi$ on $C$ that vanishes
at $t_0$. Then, since $\bfk$ has characteristic zero, one can find a translation $s$ of $C$ of infinite order such that the orbit $\{s^n(t_0):n\in\Z\}$ does not contain any other zero or pole of $\varphi$ (here we use that the characteristic of $\bfk$ is $0$). Consider the birational transformation 
$f\in \Bir(X)$ given by $f(t,x)=(s(t),\varphi(t)x)$. Let $H$ be the hypersurface $\{t_0\}\times C$. 
Then for $n\in\Z$, we have $(f_\bullet)^nH\in\Hy(X)$ if and only if $n\le 0$. Hence the action of the cyclic group~$\langle f\rangle$ does not 
transfix $\Hy(X)$.

Let us now prove that (\ref{fgnt}) implies (\ref{kod21}). Applying Theorem \ref{birtran}, and changing $X$ to a birationally equivalent surface if necessary, we assume that $X=C\times\mathbb{P}^1$ for some (smooth irreducible) curve $C$. We may now assume that the genus of $C$ is $\ge 2$, or $\ge 1$ in positive characteristic, and we have to show that every finitely generated group $\Gamma$ of $\Bir(X)$ transfixes $\Hy(X)$. 
Since the genus of $C$ is $\geq 1$,  the group $\Bir(X)$ preserves the fibration $X\to C$; this gives a surjective homomorphism 
$\Bir(X)\to\Aut(C)$. Now let us fully use the assumption on $C$: if its genus is $\ge 2$, then $\Aut(C)$ is finite; if its genus is 1 and $\bfk$ has positive characteristic, then $\Aut(C)$ is locally finite (every finitely generated subgroup is finite), and in particular the projection of $\Gamma$ on $\Aut(C)$ has a finite image.
Thus the kernel of this homomorphism intersects $\Gamma$
in a finite index subgroup $\Gamma_0$. 
It now suffices to show that $\Gamma_0$ transfixes $\Hy(X)$.
Every $f\in\Gamma_0$ has the form $f(t,x)= (t,\varphi_t(x))$ for some rational map $t\mapsto\varphi_t$ from $C$ to $\PGL_2$; 
define $U_f\subset  C$ as the open and dense subset on which $\varphi_\gamma$ is regular: by definition, $f$ restricts
to an automorphism of $U_f\times \P^1$.
Let $S$ be a finite generating subset of $\Gamma_0$, and let $U_S$ be the intersection of the open subsets $U_g$, for $g\in S$. 
Then $\Gamma_0$ acts by automorphisms on $U_S\times \P^1$ and its action on $\Hy(X)$ fixes the subset $\Hy(U_S)$.
Hence $\Gamma$ transfixes $\Hy(X)$.
\end{proof}

It would be interesting to obtain characterizations of the same properties in dimension 3 (see Question \ref{ruled3}).

\subsection{Transfixing Jonqui\`eres twists}\label{par:AppendixII}
Let $X$ be an irreducible normal projective surface and $\pi$ a morphism onto a smooth projective curve $C$
with connected rational fibers. 
Let $\Bir_\pi(X)$ be the subgroup of $\Bir(X)$ permuting the fibers of $\pi$. Since $C$ is a smooth projective curve, the group $\Bir(C)$ coincides with $\Aut(C)$ and we get a canonical homomorphism ${\mathrm{r}}_C\colon \Bir_\pi(X)\to\Aut(C)$.

The main examples to keep in mind  are provided by $\P^1\times \P^1$, Hirzebruch surfaces, and $C\times \P^1$ for
some genus~$1$ curve $C$, $\pi$ being the first projection.

Let $\Hy_\pi(X)$ denote the set of irreducible curves which are contained in fibers of $\pi$, and define $\Hyp_\pi(X)=\Hy_\pi(X)\sqcup (\Hyp(X)\smallsetminus\Hy(X))$, so that
$\Hyp(X)=\Hyp_\pi(X)\sqcup (\Hy(X)\smallsetminus\Hy_\pi(X)).$ An irreducible curve $H\subset X$ is an element of $\Hy(X)\smallsetminus\Hy_\pi(X)$ if and only if its projection $\pi(H)$ coincides with $C$; this curves are said to be transverse to $\pi$.

\begin{pro}\label{decomfi} 
The  decomposition $\Hyp(X)=\Hyp_\pi(X)\sqcup (\Hy(X)\smallsetminus\Hy_\pi(X))$ is $\Bir_\pi(X)$-invariant.
\end{pro}
\begin{proof}
Let $H\subset X$ be an irreducible curve which is transverse to $\pi$. Since $\Bir_\pi(X)$ acts by automorphisms on $C$, 
$H$ can not be contracted by any element of $\Bir_\pi(X)$; more precisely, for every $g\in \Bir_\pi(X)$, $g_\bullet(H)$ is
an element of $\Hy(X)$ which is transverse to $\pi$. Thus the set of transverse curves is $\Bir_\pi(X)$-invariant.
\end{proof}

This proposition and the proof of Theorem~\ref{thm:fin-gen-transfix} lead to the following corollary.

\begin{cor}\label{cortran}
Let $G$ be a subgroup of $\Bir_\pi(X)$. If $\pi$ maps the set of indeterminacy points of the 
elements of $G$ into a finite subset of $C$, then $G$ transfixes $\Hy(X)$. 
\end{cor}

In the case of cyclic subgroups, we establish a converse under the mild assumption of algebraic stability. Recall that a
birational transformation $f$ of a smooth projective surface is {\bf{algebraically stable}} if the forward orbit 
of $\Ind(f^{-1})$ does not intersect $\Ind(f)$. 
By \cite{Diller-Favre},  given any birational 
transformation $f$ of a surface $X$, there is a birational morphism $u\colon Y\to X$, with $Y$ a smooth projective
surface, such that $f_Y:=u^{-1}\circ f\circ u$ is algebraically stable.  If $\pi\colon X\to C$ is a fibration, as above, and 
$f$ is in $\Bir_\pi(X)$, then $f_Y$ preserves the fibration $\pi\circ u$.
Thus, we may always assume that $X$ is smooth and $f$ is algebraically stable after a birational conjugacy. 

\begin{pro}\label{protrahy}
Let $X$ be a smooth projective surface, and $\pi\colon X\to C$ a rational fibration. 
If $f\in \Bir_\pi(X)$ is algebraically stable, then $f$ transfixes $\Hy(X)$ if, and only if the orbit of $\pi(\Ind(f))$ under
the action of ${\mathrm{r}}_C(f)$ is finite. \qed
\end{pro}

For $X=\mathbb{P}^1\times\mathbb{P}^1$, the reader can check (e.g., conjugating a suitable automorphism) 
that the proposition fails without the algebraic stability assumption.

\begin{proof} Denote by $A\subset \Aut(C)$ the subgroup generated by ${\mathrm{r}}_C(f)$.
 Consider a fiber $F\simeq \P^1$ which is contracted to a point $q$ by $f$. 
Then, there is a unique indeterminacy point $p$ of $f$ on $F$. 
If the orbit of $\pi(q)$ under the action of $A$ is infinite,  the orbit of $q$ under the action of $f$ is infinite too. 
Set $q_n=f^{n-1}(q)$ for $n\geq 1$ (so that $q_1=q$); this sequence of points is well defined because 
$f$ is algebraically stable: for every $n\geq 1$, $f$ is a local isomorphism from a neighborhood of $q_n$
to a neighborhood of $q_{n+1}$. Then, the image of $F$ in $\Hyp(X)$ under the action of 
$f^n$ is an element of $\Hyp(X)\smallsetminus \Hy(X)$: it is obtained by 
a finite number of blow-ups above $q_n$.  Since the points $q_n$  
form an infinite set, the images of $F$ form an infinite subset of $\Hyp(X)\smallsetminus \Hy(X)$. 
Together with the previous corollary, this argument proves the proposition. \end{proof}

\begin{eg} \label{gatransfix}Consider $X=\P^1\times\P^1$, with $\pi(x,y)=x$  (using affine coordinates).
Start with $f_a(x,y)=(ax,xy)$, for some non-zero parameter $a\in \bfk$. 
The action of ${\mathrm{r}}_C(f_a)$ on  $C=\P^1$ fixes the images $0$ and $\infty$ of the indeterminacy points 
of $f_a$. Thus, $f_a$ transfixes $\Hyp(X)$ by Corollary \ref{cortran}. Now, consider $g_a(x,y)=(a x, (x+1)y)$. 
Then, the orbit of $-1$ under multiplication by $a$ is finite if and only if $a$ is a root of unity; thus, if $a$ is not
a root of unity, $g_a$ does not transfix $\Hy(X)$.  Section~\ref{par:example-intro6}
provides more examples of that kind. \end{eg}

\section{Birational transformations of surfaces I}

From now on, we work in dimension $2$. We shall repeatedly use two specific features of surfaces. 
First, the resolution of singularities is available in all
characteristics, so that we can always assume the varieties to be smooth. Hence $X$, $Y$, and $Z$ will be  smooth projective surfaces over 
the algebraically closed field $\bfk$. Second, smooth rational curves of self-intersection $-1$, also called exceptional 
curves of the first kind or $(-1)$-curves, can be blown down onto a smooth point. And if a curve is contracted by a birational 
morphism $\pi\colon Y\to X$, then the contraction can be down by successively contracting $(-1)$-curves.

\subsection{Regularization}
In this section, we refine Theorem~\ref{thm:FW-pseudo-regularization}, in order to apply results of 
Danilov and Gizatullin. Recall that a curve $C$ in a smooth surface $Y$ has {\bf normal crossings}
if each of its singularities is a simple node with two transverse tangents. 
In the complex case, this means that $C$ is locally analytically equivalent to 
$\{xy=0\}$ (two branches intersecting transversally) in an analytic neighborhood of each of its singularities.

\begin{thm}\label{thm:FW-regularization-surfaces}
Let $X$ be a smooth projective surface, defined over an algebraically closed field $\bfk$. Let
$\Gamma$ be a subgroup of $\Bir(X)$ that transfixes the subset $\Hy(X)$ of $\Hyp(X)$. 
There exists a smooth projective surface $Z$, a birational map $\varphi\colon Z\dasharrow X$ and 
a dense open subset $\U\subset Z$ such that, writing the boundary $\partial Z:= Z\smallsetminus \U$ 
as a finite union of irreducible components $E_i\subset Z$, $1\leq i\leq \ell$, the following properties hold:
\begin{enumerate}
\item\label{mc2} The boundary $\partial Z$ is a curve with normal crossings.
\item\label{mc3} The subgroup $\Gamma_Z:=\varphi^{-1}\circ \Gamma\circ \varphi\subset\Bir(Z)$ acts by automorphisms on the open subset $\U$.
\item\label{mc4} For all $i\in\{1,\dots,\ell\}$ and $g\in \Gamma_Z$, the strict transform of $E_i$ under the action of $g$ on $Z$ is contained in $\partial Z$: either 
$g_\circ(E_i)$ is a point of $\partial Z$ or $g_\circ(E_i)$ is an irreducible component $E_j$ of $\partial Z$. 
\item\label{mc5} For all $i\in\{1,\dots,\ell\}$, there exists an element $g\in \Gamma_Z$ that contracts $E_i$ to a point $g_\circ(E_i)\in \partial Z$. In particular, 
$E_i$ is a rational curve.
\item\label{mc6} The  pair $(Z, \U)$ is minimal for the previous properties, in the following sense: if one contracts a smooth curve of self-intersection $-1$ in $\partial Z$, then 
the boundary stops to be a normal crossing divisor.
\end{enumerate}
\end{thm}

Before starting the proof, note that the boundary $\partial Z$ may a priori contain an irreducible rational 
curve $E$ with a node. 

\begin{proof}
We apply Theorem~\ref{thm:FW-pseudo-regularization}, 
and get a smooth surface $Y_0$, a birational morphism $\varphi_0\colon Y_0\to X$, and an open subset
$\U_0$ of $Y_0$ such that  
$\Gamma_0:=\varphi_0^{-1}\circ \Gamma\circ \varphi_0$ acts by pseudo-automorphisms
on $\U_0$ and $\partial Y_0:=Y_0\setminus \U_0$ is a curve. The action of $\Gamma_0$ on $\U_0$ is not yet  by automorphisms;
we shall progressively modify the triple $(Y_0, \U_0, \varphi_0)$ to obtain a surface $Z$ with properties (\ref{mc2}) to (\ref{mc6}). 

\vspace{0.1cm}
{\bf{Step 1.--}} First, we blow-up the singularities of the curve $\partial Y_0$ which are not simple nodes to get a boundary that is 
a normal crossing divisor. This replaces the surface $Y_0$ by a new one, still denoted $Y_0$. This modification adds new 
components to the boundary $\partial Y_0$ but  does not change the fact that $\Gamma_0$ acts by pseudo-automorphisms on $\U_0$. 
Let $\ell_0$ be the number of irreducible components of $Y_0\smallsetminus\U_0$. 

\vspace{0.1cm}
{\bf{Step 2.--}}  Consider a point $q$ in $\U_0$, and assume that there is a curve $E_i$ of $\partial Y_0$ 
that is contracted to $q$ by an element $g\in \Gamma_0$; fix such a $g$, and denote by $D$ the union of 
the curves $E_j$ such that $g_\circ(E_j)=q$. By construction, $g$ is a pseudo-automorphism of~$\U_0$. 
The curve $D$ does not intersect the indeterminacy set of $g$, since otherwise there would be a curve $C$ 
containing $q$ that is contracted by $g^{-1}$. And $D$ is a connected component of $\partial Y_0$, because 
otherwise $g$ maps one of the $E_j$ to a curve that intersects~$\U_0$. 
Thus, there are open neighborhoods $\W$ of $D$ and $\W'$ of $q$ such that $\W\cap \partial Y_0=D$ and $g$ 
realizes an isomorphism from 
$\W\smallsetminus D$ to $\W'\smallsetminus \{ q\}$, contracting $D$ onto the smooth point $q\in Y_0$. 
In particular, $\W$ can be contracted onto a smooth point (by a succession of contractions of exceptional curves of the first kind). 
As a consequence, 
there is a birational morphism $\pi_1\colon Y_0\to Y_1$ 
such that 
\begin{enumerate} 
\item $Y_1$ is smooth 
\item $\pi_1$ contracts $D$ onto a point $q_1\in Y_1$ 
\item  $\pi_1$ is an isomorphism from $Y_0\smallsetminus D$ to 
$Y_1\smallsetminus \{q_1\}$. 
\end{enumerate}
In particular, $\pi_1(\U_0)$ is an open subset of $Y_1$ and $\U_1=\pi_1(\U_0) \cup \{q_1\}$ is an open 
neighborhood of $q_1$ in $Y_1$. 

Then, $\Gamma_1:=\pi_1\circ \Gamma_0\circ \pi_1^{-1}$ acts birationally on $Y_1$, and by pseudo-automorphisms on $\U_1$. The
boundary $\partial Y_1=Y_1\smallsetminus \U_1$ contains $\ell_1$ irreducible components, with $\ell_1< \ell_0$ (the difference is the number
of components of $D$), and is a normal crossing divisor because $D$ is a connected component of $\partial Y_0$. 

Repeating this process, we construct a sequence of surfaces $\pi_k\colon Y_{k-1}\to Y_{k}$ and open 
subsets $\pi_k(\U_{k-1})\subset \U_k\subset Y_k$ such that the number of irreducible 
components of $\partial Y_k=Y_k\smallsetminus \U_k$ decreases.
After a finite number of steps (at most $\ell_0$), we may assume that $\Gamma_k\subset\Bir(Y_k)$ does not 
contract any boundary curve onto a point of the open subset 
$\U_k$. On such a model, $\Gamma_k$ acts by automorphisms on $\U_k$. 

We fix such a model, which we denote by the letters $Y$, $\U$, $\partial Y$, $\varphi$. The new birational map 
$\varphi\colon Y\dasharrow X$ is the composition of $\varphi_0$ with the inverse of the morphism $Y_0\to Y_k$. On such a 
model, properties  (\ref{mc2}) and (\ref{mc3}) are satisfied. Moreover, (\ref{mc4}) follows from (\ref{mc3}). We now modify $Y$ further to get property (\ref{mc5}).

\vspace{0.1cm}
{\bf{Step 3.--}} Assume that the curve $E_i\subset Y\smallsetminus \U$ is not contracted by $\Gamma$. 
Let $F$ be the orbit of $E_i$: $F=\cup_{g\in\Gamma} g_\circ(E_i)$; by property (\ref{mc4}),
this curve is contained in the boundary $\partial Y$ of the open subset $\U$. Let $\overline{\partial Y\smallsetminus F}$ denote
the Zariski closure of $\partial Y\smallsetminus F$, and set 
\[
\U'= \U \cup (F\smallsetminus \overline{\partial Y\smallsetminus F}).
\]
The group $\Gamma$ also acts by pseudo-automorphisms on $\U'$. This operation decreases the number $\ell$ 
of irreducible components of the boundary. Thus, combining steps 2 and 3 finitely many times, we reach a model that
satisfies Properties (\ref{mc2}) to (\ref{mc5}). We continue to denote it by $Y$.

\vspace{0.1cm}
{\bf{Step 4.--}} If the boundary $\partial Y $ contains a smooth (rational) curve $E_i$ of self-intersec\-tion~$-1$, 
 it can be blown down to a smooth point $q$ by a birational morphism 
$\pi\colon Y\to Y'$; the open subset $\U$ is not affected, but the boundary $\partial Y'$ has one component less. 
If $E_i$ was a connected component of $\partial Y$, then $\U'=\pi(\U)\cup \{q\}$ is a neighborhood
of $q$ and one replaces $\U$ by $\U'$, as in step 2.
Now, two cases may happen. If the boundary $\partial Y'$
ceases to be a normal crossing divisor, we come back to $Y$ and do not apply this surgery. If $\partial Y'$ has normal crossings, 
we replace $Y$ by this new model. In a finite number of steps, looking successively at all $(-1)$-curves and iterating the 
process, we reach a new surface $Z$ on which all five properties are satisfied. 
\end{proof}

\begin{rem} One may also remove property (\ref{mc6}) and replace property (\ref{mc2}) by 
\begin{itemize}
\item[(1')] The $E_i$ are rational curves, and none of them is a smooth rational curve with self-intersection $-1$.
\end{itemize}
But doing so, we may lose the normal crossing property. To get property (1'), apply the theorem and argue as 
in step 4. \end{rem}

\subsection{Constraints on the boundary}\label{par:constraints-boundary}

We now work on the new surface $Z$ given by Theorem~\ref{thm:FW-regularization-surfaces}. 
Thus, $Z$ is the surface, $\Gamma$ the subgroup of $\Bir(Z)$, $\U$ the open subset on which $\Gamma$ 
acts by automorphisms,  and $\partial Z$ 
the boundary of $\U$. 

\begin{pro}[Gizatullin, \cite{Gizatullin:1971} \S~4]\label{pro:Gizatullin-boundary}
There are four possibilities for the geometry of the boundary $\partial Z=Z\smallsetminus \U$. 
\begin{enumerate}
\item\label{giz1}  $\partial Z$ is empty.
\item\label{giz2}  $\partial Z$ is a cycle of rational curves.
\item\label{giz3} $\partial Z$ is a chain of $\ell$ rational curves and if $\ell=1$ it is a smooth rational curve of positive self intersection.
\item\label{giz4} $\partial Z$ is the disjoint union of finitely many smooth rational curves of self-intersection $0$.
\end{enumerate}
Moreover, in cases (\ref{giz2}) and (\ref{giz3}), the open subset $\U$ is the blow-up of an affine surface.
\end{pro}

Thus, there are four possibilities for $\partial Z$, which we study successively. 
We shall start with (1) and (4) in sections~\ref{par:Projective-Surfaces} 
and~\ref{par:invariant-fibrations}. Then case (3) is dealt with in Section~\ref{par:completion-zigzag}. 
Case (2) is  more involved: it is treated in Section~\ref{par:Cycles-Thompson}.

Before that, let us explain how Proposition~\ref{pro:Gizatullin-boundary} follows from Section~5 of \cite{Gizatullin:1971}. 
First, we describe the precise meaning of the statement, and
then we explain how the original results of~\cite{Gizatullin:1971} apply to our situation.

\vspace{0.2cm}
{\bf{The boundary and its dual graph .}}-- Consider the dual graph $\G_Z$ of the boundary $\partial Z$. The vertices
of $\G_Z$ are in one to one correspondence with the irreducible components $E_i$ of $\partial Z$. The edges correspond
to singularities of $\partial Z$: each singular point $q$ gives rise to an edge connecting the components $E_i$ that determine
the two local branches of $\partial Z$ at $q$. When the two branches correspond to the same irreducible component, 
one gets a loop of the graph $\G_Z$. 

We say that $\partial Z$ is a {\bf{chain}} of rational curves if the dual graph is of type $A_\ell$: $\ell$ is the number of components, 
and the graph is linear, with $\ell$ vertices. 
Chains are also called {\bf{zigzags}} by Danilov and Gizatullin. 

We say that $\partial Z$ is a {\bf{cycle}} if the dual graph is isomorphic to a regular polygon with $\ell$ vertices. 
There are two special cases: when $\partial Z$ is reduced to one component, this curve is a rational curve with one 
singular point and the dual graph is a loop (one vertex, one edge); when $\partial Z$ is made of two components, these components 
intersect in two distinct points, and the dual graph is made of two vertices with two edges between them. For $\ell=  3, 4, \ldots$, 
the graph is a triangle, a square, etc.
 
\vspace{0.2cm}
{\bf{Gizatullin's original statement.}}-- 
To describe Gizatullin's article, let us introduce some vocabulary. Let $S$ be a projective surface, and $C\subset S$ 
be a curve; $C$ is a union of irreducible components, which may have singularities. Assume that $S$ is smooth in a neighborhood
of $C$. Let $S_0$ be the complement of $C$ in $S$, and let $\iota \colon S_0\to S$ be the natural embedding of $S_0$ in $S$. 
Then, $S$ is a {\bf{completion}} of $S_0$: this completion is marked by the embedding $\iota \colon S_0\to S$, and its boundary is 
the curve $C$. Following \cite{Gizatullin:1971} and \cite{Danilov-Gizatullin:I,Danilov-Gizatullin:II}, we only consider completions 
of $S_0$ by curves (i.e.\ $S\smallsetminus \iota(S_0)$ is of pure dimension $1$), and we always assume $S$ to be smooth in a
neighborhood of the boundary. Such a completion is  

\begin{itemize}
\item[(i)] {\bf{simple}} if the boundary $C$ has normal crossings;
\item[(ii)] {\bf{minimal}} if it is simple and minimal for this property: if $C_i\subset C$ is an exceptional 
curve of the first kind then, contracting $C_i$, the image of $C$ is not a normal  crossing divisor
anymore. Equivalently, $C_i$ intersects at least three other components of $C$. Equivalently, if $\iota'\colon S_0\to S'$
is another simple completion, and $\pi\colon S\to S'$ is a birational morphism such that $\pi\circ \iota=\iota'$, then $\pi$
is an isomorphism. 
\end{itemize}
If $S$ is a completion of $S_0$, one can blow-up boundary points to obtain a simple completion, and then blow-down some
of the boundary components $C_i$ to reach a minimal completion. 

Now, consider the group of automorphisms of the open surface $S_0$. This group $\Aut(S_0)$ acts by birational transformations
on $S$. An irreducible component $E_i$ of the boundary $C$ is {\bf{contracted}} if there is an element $g$ of $\Aut(S_0)$ that 
contracts $E_i$: $g_\circ(E_i)$ is a point of $C$. Let $E$ be the union of the contracted components. In \cite{Gizatullin:1971} (Corollary 4 and Proposition 5 of \S 5), Gizatullin 
proves that $E$ satisfies one of the four properties stated in Proposition~\ref{pro:Gizatullin-boundary}; moreover, in cases (2) and (3), 
$E$ contains an irreducible component $E_i$ with $E_i^2 \geq 0$; note that 
(4) contains the case of a unique rational curve of self-intersection $0$ (a different choice is made in \cite{Gizatullin:1971}). 

Thus, Proposition~\ref{pro:Gizatullin-boundary}  follows from the properties of the pair $(Z,\U,\Gamma)$: the open subset $\U$ plays the
role of $S_0$, and $Z$ is the completion $S$; the boundary $\partial Z$ is the curve $C$: it is a normal crossing divisor, and it
is minimal by construction. Since every component of $\partial Z$ is contracted by at least one element of $\Gamma\subset \Aut(\U)$, 
$\partial Z$ coincides with Gizatullin's curve $E$. The only thing we have to prove is the last sentence of Proposition~\ref{pro:Gizatullin-boundary}, concerning the structure of the open subset $\U$; thus, we assume that we are in cases (2) or (3) of Proposition~\ref{pro:Gizatullin-boundary}.

First, let us show that $E=\partial Z$ supports an effective divisor $D$ such that $D^2>0$ and $D\cdot F \geq 0$ for every irreducible curve.
If $\partial Z$ is irreducible, then it is a curve of positive self intersection (by convention in case (3), and by Corollary 4 in~\cite[\S 4]{Gizatullin:1971}).
Thus, we may assume that $\partial Z$ is a chain or a loop of length $\ell \geq 2$.
To construct $D$, fix an irreducible component $E_0$ of $\partial Z$ with $E_0^2\geq 0$; as said above such a curve exists by Gizatullin results (Corollary 4 of \cite[\S 5]{Gizatullin:1971}). Assume that $\partial Z$ is a  cycle, and list cyclically the other irreducible components: $E_1$, $E_2$, $\ldots$, up to $E_m$, with $E_1$ and $E_m$ intersecting $E_0$ (and $m=\ell -1$). 
If $m=1$, we set $D=a_0 E_0+E_1$; then $D\cdot E_0=2$ and $D\cdot E_1=2a_0+E_1^2$ are positive if $a_0$ is large enough. If $m\geq 2$, 
we consider $D_1=a_1 E_0+E_1$. Then $D_1\cdot E_0=1$ and $D_1\cdot E_1=a_1+E_1^2$ are both positive if $a_1$ is large enough; moreover, $D_1\cdot E_2=1$ and $D_1\cdot E_m=a_1$ if $m\geq 3$, or $D_1\cdot E_2=a_1+1$ if $m=2$. Then, set $D_2=a_2D_1+E_2$, $\ldots$,  
up to $D_m=a_m D_{m-1}+E_m$. If the $a_j$ are large enough, all intersections $D_m\cdot E_j$ are positive, for all $0\leq j \leq m$. We choose such a
sequence of integers $a_i$, and set $D=D_m$. Then  $D$ intersects every   irreducible curve $F$ non-negatively and $D^2>0$. Thus, $D$ is big and nef (see \cite{Lazarsfeld}, Section~2.2). 
 A similar proof applies when $\partial Z$ is a zigzag. Let $[D]^\perp$ be the subspace of
$\NS(X)$ spanned by classes of irreducible curves $F$ with $D\cdot F=0$. 

Now, consider the linear system $\vert mD\vert$ for a large divisible integer $m>0$, and decompose it into a mobile part
$\vert M_m\vert$ and a fixed part $\vert R_m\vert$, where $M_m$ and $R_m$ are effective divisors with
\[
mD=M_m+R_m.
\]
Note that the irreducible curves $F$ with $[F]\in [D]^\perp$ are characterized by the property $F\cdot M_m=0=F\cdot R_m$.
By definition, $\vert M_m\vert$ has only finitely many base points. Thus, changing $m$ into some large multiple if necessary, 
and applying Fujita-Zariski theorem (see~\cite{Lazarsfeld}, 2.1.32, p. 132), we may assume that 
\begin{itemize}
\item[(i)] $M_m$ is big (because so is $D$);
\item[(ii)] $M_m$ is nef (because $M_m$ is mobile);
\item[(iii)] $M_m$ is free of base point (by Fujita-Zariski theorem).
\end{itemize}
Then, the linear system $\vert M_m\vert$ gives a birational morphism (see~\cite{Lazarsfeld}, 2.1.27, p. 129)
\[
\varphi\colon Z\to Z'\subset {\mathbb{P}}^N
\]
onto a normal, projective surface $Z'$ such that $M_m$ coincides with the pullback of a hyperplane section $H$
of $Z'$. In particular, $H^1(Z,dM_m)=0$ for large values of $d$. Now,  let us show that $R_m=0$ for some adequate
choice of $mD$. If not, some curve $E_j$ of the boundary $\partial Z$ appears as a component of $R_m$ but not as
a component of $M_m$; since $\partial Z={\mathrm{Support}}(D)$ is connected, we can choose such an $E_j$ that intersects $M_m$
in at least one point. Thus, $(dM_m+a_jE_j)\cdot E_j>0$ for any $a_j$ and every large $d$. Consider the 
exact sequence of sheaves ${\mathcal{O}}(dM_m)\to {\mathcal{O}}(dM_m+ E_j)\to {\mathcal{O}}_{E_j}((dM_m+ E_j)_{\vert E_j})$,
and the associated long exact sequence in cohomology. By the vanishing of $H^1(Z,dM_m)$ we get 
\[
H^0(Z, dM_m)\to H^0(Z, dM_m+E_j)\to H^0(E_j, dM_m+E_j)\to 0.
\]
If $E_j$ were part of the base locus of the linear system $\vert dM_m+E_j\vert$, then the second arrow in this sequence would vanish, 
so that $ H^0(E_j, dM_m+E_j)=0$. But this would be a contradiction because $E_j$ is a rational curve and $dM_m+E_j$ has positive
degree on $E_j$. Thus, $E_j$ is not in the base locus of $\vert dM_m +E_j\vert$: we may now assume $R_m=0$. 
From this, we deduce that an irreducible curve $C\subset Z$ is
contracted by $\varphi$, if and only if $C\cdot M_m=0$ if and only if $[C]\in [D]^\perp$, if and only if $C$ does not intersect the
boundary curve $\partial Z$; and that $\varphi$ induces a birational morphism from $Z\setminus \partial Z$ to the
affine surface $Z'\setminus H$. This concludes the proof of the proposition. 

\subsection{Projective surfaces and automorphisms}\label{par:Projective-Surfaces}

In this section, we (almost always) assume that $\Gamma$ acts by regular automorphisms on a projective surface $X$. This corresponds
to case (1) in Proposition~\ref{pro:Gizatullin-boundary}. Our goal is the special case of Theorem~B which is stated below 
as Theorem~\ref{thm:classification-virtually-autom}. We shall assume that $\Gamma$ has property {\FW} in some of the statements
(this was not a hypothesis in Theorem~\ref{thm:FW-regularization-surfaces}).
We may, and shall, assume that $X$ is smooth. We refer to~\cite{BPVDVH, Beauville:Surfaces, Hartshorne:book} for the classification of surfaces 
and the main notions attached to them.
 
\subsubsection{Action on the N\'eron-Severi group}
The intersection form is a non-degenerate quadratic form $q_X$ on the N\'eron-Severi group $\NS(X)$, 
and Hodge index theorem asserts that its signature is $(1, \rho(X)-1)$, where $\rho(X)$ denotes the Picard number, i.e.\ the rank of
the lattice $\NS(X)\simeq \Z^\rho$. 

The action of $\Aut(X)$ on the N\'eron-Severi group $\NS(X)$ provides a linear representation preserving
the intersection form $q_X$. This gives a morphism
\[
\Aut(X)\to \O(\NS(X); q_X).
\]
Fix an ample class $a$ in $\NS(X)$ and consider the hyperboloid 
\[
{\mathbb{H}}_X=\{ u \in \NS(X)\otimes_\Z\R; \; q_X(u,u)=1 \; \,  {\text{and}}\, \; q_X(u,a)>0 \}.
\]
This set is one of the two connected components of $\{u; q_X(u,u)=1\}$. With the riemannian metric 
induced by $(-q_X)$, it is a copy of the hyperbolic space of dimension $\rho(X)-1$; the group $\Aut(X)$ 
acts by isometries on this space (see \cite{Cantat:Milnor}). 

\begin{pro}\label{pro:BHW-Automorphisms}
Let $X$ be a smooth projective surface. Let $\Gamma$ be a subgroup of $\Aut(X)$. If $\Gamma$ 
has Property {\FW}, then its action on $\NS(X)$ fixes a very ample class, the image of $\Gamma$ in 
$\O(\NS(X); q_X)$ is finite, and a finite index subgroup of $\Gamma$ is contained in $\Aut(X)^0$.
\end{pro}

\begin{proof}
The image $\Gamma^*$ of $\Gamma$ is contained in the arithmetic group $\O(\NS(X); q_X)$. The 
N\'eron-Severi group $\NS(X)$ is a lattice $\Z^\rho$ and $q_X$ is defined over $\Z$. If $\rho$ is odd, 
one can change $\NS(X)$ into a $(\rho+1)$-dimensional lattice $\NS(X)\oplus \Z{\mathbf{e}}$ and change $q_X$ into the quadratic form defined
by $q({\mathbf{u}}+m{\mathbf{e}})=q_X({\mathbf{u}})-m^2$ for all ${\mathbf{u}}+m{\mathbf{e}}$ in $\NS(X)\oplus \Z{\mathbf{e}}$.
After such a change, $\Gamma^*$ embeds into the orthogonal group $\O(\Z^r; q)$ for some even $r\in \{\rho,\rho+1\}$ and some integral
quadratic form of signature $(1,r-1)$. It is proved by Bergeron, Haglund, and Wise 
that such a lattice acts properly on some ${\text{\sc{cat}}}(0)$ cube complex (see Theorem 6.1 and the paragraph before Theorem 6.2 in \cite{Bergeron-Wise}; see \cite{Bergeron-Haglund-Wise:2011} for the case of uniform lattices). But if a group with Property {\FW} acts by isometries on such a complex, it has a fixed point (see \cite{Cornulier:Survey-FW}). Thus, by properness of
the action, the image $\Gamma^*$ of $\Gamma$ in $\O(\NS(X);  q_X)$ is finite. 

The kernel $K\subset \Aut(X)$ of the action on $\NS(X)$ contains $\Aut(X)^0$ as a finite index subgroup. Thus, 
if $\Gamma$ has Property {\FW}, it contains a finite index subgroup that is contained in $\Aut(X)^0$ (see Theorem~\ref{thm:pseudo-automorphisms-neron-severi}). \end{proof}

\subsubsection{Non-rational surfaces}\label{par:non-rational-surfaces}

Here, the surface $X$ is not rational. The following proposition classifies subgroups 
of $\Bir(X)$ with Property {\FW}; in particular, such a group is finite if the Kodaira
dimension of $X$ is non-negative (resp.\ if the characteristic of $\bfk$ is positive). Recall that 
$\overline{\Z}\subset \overline{\Q}$ is the ring of algebraic integers.

\begin{pro}
Let $X$ be a smooth, projective, and  non-rational surface, over the algebraically closed field $\bfk$. 
Let $\Gamma$ be an infinite subgroup of $\Bir(X)$ with Property~{\FW}. Then $\bfk$ has characteristic
$0$, and there is a birational map $\varphi\colon X\dasharrow C\times \P^1_\bfk$ 
that conjugates $\Gamma$ to a subgroup of $\Aut(C\times \P^1_\bfk)$. Moreover, there is
a finite index subgroup $\Gamma_0$ of $\Gamma$ such that $\varphi\circ \Gamma_0\circ \varphi^{-1}$, 
is a subgroup of $\PGL_2(\oZ)$, acting
on $C\times \P^1_\bfk$ by linear projective transformations on the second factor. \end{pro}

\begin{proof}
Assume, first, that the Kodaira dimension of $X$ is non-negative. Let $\pi\colon X\to X_0$ be the projection of 
$X$ on its (unique) minimal model (see \cite{Hartshorne:book}, Thm. V.5.8). 
The group $\Bir(X_0)$ coincides with $\Aut(X_0)$; thus, after conjugacy by $\pi$, $\Gamma$
becomes a subgroup of $\Aut(X_0)$, and Proposition~\ref{pro:BHW-Automorphisms} provides a finite index
subgroup  $\Gamma_0\leq \Gamma$ that is contained in $\Aut(X_0)^0$. Note that $\Gamma_0$ inherits Property {\FW} from $\Gamma$.
 
If the Kodaira dimension of $X$ is equal to $2$, the group $\Aut(X_0)^0$ is trivial; hence $\Gamma_0=\{\Id_{X_0}\}$ and $\Gamma$ 
is finite. 
If the Kodaira dimension is equal to $1$, $\Aut(X_0)^0$ is either trivial, or isomorphic to an elliptic curve, acting by translations on 
the fibers of the Kodaira-Iitaka fibration of $X_0$ (this occurs, for instance, when $X_0$ is the product of an elliptic curve with a curve of higher genus). 
If the Kodaira dimension is $0$, then $\Aut(X_0)^0$ is also an abelian group (either trivial, or isomorphic to an abelian surface). 
Since abelian groups with Property {\FW} are finite, the group $\Gamma_0$ is finite, and so is $\Gamma$. 

We may now assume that the Kodaira
dimension ${\sf{kod}}(X)$ is negative. Since $X$ is not rational, then $X$ is birationally equivalent to a product $S=C\times \P^1_\bfk$, 
where $C$ is a curve of genus ${\mathrm{g}}(C)\geq 1$.  Denote by $\bfk(C)$ the field of rational functions on the curve~$C$. 
The semi-direct product $\Aut(C)\ltimes \PGL_2(\bfk(C))$ acts on $S$ by birational transformations of the form
\[
(x,y)\in C\times \P^1_\bfk \mapsto \left(f(x), \frac{a(x)y+b(x)}{c(x)y+d(x)}\right);
\]
here $f$ is an automorphism of $C$, and $a$, $b$, $c$, and $d$ are elements of $\bfk(C)$ such that $ad-bc$ is not identically $0$.
Moreover, $\Bir(S)$ coincides with this group $\Aut(C)\ltimes \PGL_2(\bfk(C))$ because 
the first projection $\pi\colon S \to C$ is equivariant under the action of $\Bir(S)$ (this follows from the fact that every rational map $\P^1_\bfk\to C$ is constant).

Since ${\mathrm{g}}(C)\geq 1$, $\Aut(C)$ is virtually abelian. Property {\FW} implies that  there is a finite index, normal subgroup 
$\Gamma_0\leq \Gamma$ that is contained in $\PGL_2(\bfk(C))$. By
Corollary~\ref{coro:Bass-k(C)}, every subgroup of $\PGL_2(\bfk(C))$ with Property {\FW} is conjugate to a subgroup of $\PGL_2(\oZ)$ or a finite group if the characteristic of the field $\bfk$ is positive. 

We may assume now that the characteristic of $\bfk$ is $0$ and that $\Gamma_0\subset \PGL_2(\oZ)$ is infinite. Consider an element $g$ of $\Gamma$; it
acts as a birational transformation on the surface $S=C\times \P^1_\bfk$, and  it normalizes $\Gamma_0$: 
\[
g\circ \Gamma_0=\Gamma_0 \circ g.
\]
Since $\Gamma_0$ acts by automorphisms on $S$, the finite set $\Ind(g)$ is $\Gamma_0$-invariant. But a subgroup of $\PGL_2(\bfk)$ with 
Property {\FW} preserving a non-empty, finite subset of $\P^1(\bfk)$ is a finite group (by Lemma~\ref{lem:pgl2-fw}(2)). Thus, $\Ind(g)$ must be empty. This shows that
$\Gamma$ is contained in $\Aut(S)$.
\end{proof}

\subsubsection{Rational surfaces}\label{par:auto-rational-surfaces}

Now, we assume that $X$ is a smooth rational surface, that $\Gamma\leq \Bir(X)$ is an infinite subgroup with Property {\FW}, and that $\Gamma$
contains a finite index, normal subgroup $\Gamma_0$ that is contained in $\Aut(X)^0$. Recall that a smooth surface $Y$
is minimal if it does not contain any smooth rational curve of the first kind, {i.e.}\ with self-intersection $-1$.
Every exceptional curve of the first kind in $X$ is determined by its class in $\NS(X)$ and is therefore  invariant
under the action of $\Aut(X)^0$. The following lemma is obtained by contracting such $(-1)$-curves one by one.

\begin{lem}
 There is a birational morphism $\pi\colon X\to Y$ onto a
minimal rational surface $Y$ that is equivariant under the action of $\Gamma_0$; $Y$ does not contain any
exceptional curve of the first kind and $\Gamma_0$ becomes a subgroup of $\Aut(Y)^0$.
\end{lem} 

Let us recall the classification of minimal rational surfaces and describe their groups of automorphisms. First, 
we have the projective plane $\P^2_\bfk$, with $\Aut(\P^2_\bfk)= \PGL_3(\bfk)$ acting by linear projective transformations. 
Then comes the quadric $\P^1_\bfk\times \P^1_\bfk$, with 
$\Aut(\P^1_\bfk\times \P^1_\bfk)^0=\PGL_2(\bfk)\times \PGL_2(\bfk)$ 
acting by linear projective transformations on each factor; the group of automorphisms of the quadric 
is the semi-direct product of $\PGL_2(\bfk)\times \PGL_2(\bfk)$ with the group of order $2$ generated 
by the permutation of the two factors, $\eta(x,y)=(y,x)$. Then, for each integer $m\geq 1$, the Hirzebruch surface
$\Hirz_m$  is the projectivization of the rank $2$ bundle ${\mathcal{O}}\oplus {\mathcal{O}}(m)$ over $\P^1_\bfk$; 
it may be characterized as the unique ruled surface $Z\to \P^1_\bfk$ with a section $C$ of self-intersection $-m$. Its 
group of automorphisms is connected and preserves the ruling. This provides a homomorphism $\Aut(\Hirz_m)\to \PGL_2(\bfk)$
that describes the action on the base of the ruling, and it turns out that this homomorphism is surjective. If we choose coordinates for which 
the section $C$ intersects each fiber at infinity, the kernel $J_m$ of this homomorphism acts by transformations of type 
\[
(x,y)\mapsto (x,\alpha y+\beta(x))
\]
where $\beta(x)$ is a polynomial function of degree $\leq m$. In particular, $J_m$ is solvable.
In other words, $\Aut(\Hirz_m)$ is isomorphic to the group 
\[
\left( \GL_2(\bfk)/\mu_m\right)\ltimes W_m
\]
where $W_m$ is the linear representation of $\GL_2(\bfk)$ on homogeneous polynomials of degree $m$ in two variables,
and $\mu_m$ is the kernel of this representation: it is the subgroup of $\GL_2(\bfk)$ given by scalar multiplications by 
roots of unity of order dividing~$m$. 

\begin{lem}
Given the above conjugacy $\pi\colon X\to Y$, the subgroup $\pi\circ\Gamma\circ\pi^{-1}$ of $\Bir(Y)$ is contained in $\Aut(Y)$.
\end{lem}
 
\begin{proof}
Assume that the surface $Y$ is the quadric $\P^1_\bfk\times \P^1_\bfk$. Then, according to Theorem~\ref{thm:Bass}, 
$\Gamma_0$ is conjugate to a subgroup 
of $\PGL_2(\oZ)\times \PGL_2(\oZ)$. If $g$ is an element of $\Gamma$, its indeterminacy locus is a finite 
subset $\Ind(g)$ of $\P^1_\bfk\times \P^1_\bfk$ that is invariant under the action of $\Gamma_0$, because 
$g$ normalizes $\Gamma_0$. Since $\Gamma_0$ is infinite and has Property {\FW}, this set $\Ind(g)$ is empty (Lemma~\ref{lem:pgl2-fw}). Thus, 
$\Gamma$ is contained in $\Aut(\P^1_\bfk\times \P^1_\bfk)$. 

The same argument applies for Hirzebruch surfaces. Indeed, $\Gamma_0$ is an infinite subgroup of $\Aut(\Hirz_m)$ 
with Property {\FW}. Thus, up to conjugacy, its projection in $\PGL_2(\bfk)$ is contained in $\PGL_2(\oZ)$. 
If it were finite, a finite index
subgroup of $\Gamma_0$ would be contained in the solvable group $J_m$, and would therefore be finite too by Property {\FW}; 
this would contradict $\vert \Gamma_0\vert =\infty$. Thus, the projection of $\Gamma_0$ in $\PGL(\oZ)$ is infinite. 
If $g$ is an element of $\Gamma$, $\Ind(g)$  
is a finite, $\Gamma_0$-invariant subset, and by looking at the projection of this set in $\P^1_\bfk$ one 
deduces that it is empty (Lemma~\ref{lem:pgl2-fw}). This proves that $\Gamma$ is contained in $\Aut(\Hirz_m)$. 

Let us now assume that $Y$ is the projective plane. Fix an element $g$ of $\Gamma$, and assume that
$g$ is not an automorphism of $Y=\P^2$; the indeterminacy and exceptional sets of $g$ are $\Gamma_0$ invariant. 
Consider an irreducible curve $C$ in the exceptional set of $g$, together with an indeterminacy point $q$
of $g$ on $C$. Changing $\Gamma_0$ in a finite index subgroup, we may assume that $\Gamma_0$
fixes $C$ and $q$; in particular, $\Gamma_0$ fixes $q$, and permutes the tangent lines of $C$ through 
$q$. But the algebraic subgroup of $\PGL_3(\bfk)$ preserving a point $q$ and a line through
$q$ does not contain any infinite group with Property {\FW} (Lemma~\ref{lem:pgl2-fw}).
Thus, again, $\Gamma$ is contained in $\Aut(\P^2_\bfk)$.
\end{proof}

\subsubsection{Conclusion, in Case (1)} 
Putting everything together, we obtain the following particular case of Theorem~B.

\begin{thm}\label{thm:classification-virtually-autom}
Let $X$ be a smooth projective surface over an algebraically closed field $\bfk$. 
Let $\Gamma$ be an infinite subgroup of $\Bir(X)$ with Property {\FW}. If a finite index
subgroup of $\Gamma$ is contained in $\Aut(X)$, there is a birational morphism $\varphi\colon X\to Y$ that conjugates 
$\Gamma$ to a subgroup  $\Gamma_Y$
of $\Aut(Y)$, with $Y$  in the following list:
\begin{enumerate}
\item $Y$ is the product of a curve $C$ by $\P^1_\bfk$, the field $\bfk$ has
characteristic $0$, and a finite index subgroup $\Gamma'_Y$ of $\Gamma_Y$ is contained in 
 $\PGL_2(\oZ)$, acting by linear  projective transformations on the second factor;
\item $Y$ is $\P^1_\bfk\times \P^1_\bfk$, the field $\bfk$ has characteristic $0$,  and $\Gamma_Y$ is contained in $\PGL_2(\oZ)\times \PGL_2(\oZ)$;
\item $Y$ is a Hirzebruch surface $\Hirz_m$ and $\bfk$ has characteristic $0$; 
\item $Y$ is the projective plane $\P^2_\bfk$.
\end{enumerate}
In particular, $Y=\P^2_\bfk$ if the characteristic of $\bfk$ is positive.
\end{thm}

\begin{rem}
 Denote by $\varphi \colon X\to Y$ the birational 
morphism given by the theorem. Changing $\Gamma$ in a finite index subgroup, we may assume that 
it acts by automorphisms on both $X$ and $Y$.

If $Y=C\times  \P^1$, then $\varphi$ is in fact an isomorphism. To prove this fact, denote 
by $\psi$ the inverse of $\varphi$. The  indeterminacy set $\Ind(\psi)$ is $\Gamma_Y$ invariant because both $\Gamma$ and $\Gamma_Y$ act by 
automorphisms. From Lemma~\ref{lem:pgl2-fw}, applied to $\Gamma'_Y\subset \PGL_2(\bfk)$, we deduce that $\Ind(\psi)$  is empty and  $\psi$ is an isomorphism. 
The same argument implies that the conjugacy is an isomorphism if $Y=\P^1_\bfk\times \P^1_\bfk$ or a Hirzebruch surface $\Hirz_m$, $m\geq 1$. 
 
Now, if $Y$ is $\P^2_\bfk$, $\varphi$ is not always an isomorphism. For instance, $\SL_2(\C)$ acts on $\P^2_\bfk$ with a fixed
point, and one may blow up this point to get a new surface with an action of groups with Property {\FW}. But this is the only possible example, 
{\sl{i.e.}} $X$ is either $\P^2_\bfk$, or a single blow-up of $\P^2_\bfk$ (because $\Gamma\subset \PGL_3(\C)$ can not preserve more
than one base point for $\varphi^{-1}$ without loosing Property {\FW}).
\end{rem}

\subsection{Invariant fibrations}\label{par:invariant-fibrations}

We now assume that $\Gamma$ has Property~{\FW} and acts by automorphisms on $\U\subset X$, 
and that the boundary $\partial X=X\smallsetminus \U$ is the union of $\ell \geq 1$ pairwise disjoint rational curves $E_i$; each of them has self-intersection $E_i^2=0$ and is contracted
by at least one element of $\Gamma$. This corresponds to the fourth possibility in Gizatullin's Proposition~\ref{pro:Gizatullin-boundary}.
Since $E_i\cdot E_j=0$, the Hodge index theorem implies that the classes $e_i=[E_i]$ span a unique line in $\NS(X)$, and that
$[E_i]$ intersects non-negatively every curve.

From Section~\ref{par:non-rational-surfaces}, we may, and do assume that $X$ is a rational surface.
In particular, the Euler characteristic of the structural sheaf is equal to $1$: $\chi({\mathcal{O}}_X)=1$, 
and Riemann-Roch formula gives 
\[
h^0(X, E_1)-h^1(X,E_1)+h^2(X,E_1)= \frac{E_1^2-K_X\cdot E_1}{2}+1.
\]
The genus formula implies $K_X\cdot E_1=-2$, and Serre duality shows that $h^2(X, E_1)=h^0(X, K_X-E_1)=0$
because otherwise $-2=(K_X-E_1)\cdot E_1$ would be non-negative (because $E_1$ 
intersects non-negatively every curve). From this, we obtain
\[
h^0(X,E_1)= h^1(X,E_1)+2\geq 2.
\]
If $F$ is a member of the complete linear system $\vert E_1\vert$, then $F\cdot E_1=E_1\cdot E_1=0$, and $F$ is disjoint from the smooth irreducible curve $E_1$. Thus, $\vert E_1\vert$ is base point free, and $\vert E_1\vert$ determines a fibration $\pi\colon X\to B$ onto a curve $B$; in fact $B=\P^1_\bfk$ because $X$ is a rational surface, and $H^0(X,E_1)=2$ because ${\mathcal{O}}(E_1)$ is the pull back of an ample line bundle on $B$ (see Theorem 2.1.27 in \cite{Lazarsfeld}). The curve $E_1$, as well as the
$E_i$ for $i\geq 2$, are fibers of $\pi$. 

If $f$ is an automorphism of $\U$ and $F\subset \U$ is a fiber of $\pi$, then $f(F)$ is a (complete) rational curve. Its projection 
$\pi(f(F))$ is contained in the affine curve $\P^1_\bfk\smallsetminus \cup_i\pi(E_i)$ and must therefore be reduced to a point. Thus, $f(F)$ is a fiber of $\pi$
and $f$ preserves the fibration. This proves the following lemma.

\begin{lem}
There is a fibration $\pi\colon X\to \P^1_\bfk$ such that 
\begin{enumerate}
\item every component $E_i$ of $\partial X$ is a fiber of $\pi$, and  $\U=\pi^{-1}(\V)$ for 
an open subset $\V\subset \P^1_\bfk$;
\item the general fiber  of $\pi$ is a smooth rational curve;
\item $\Gamma$ permutes the fibers of $\pi$: there is a morphism $\rho\colon \Gamma\to \PGL_2(\bfk)$
such that $\pi\circ f=\rho(f)\circ \pi$ for every $f\in \Gamma$.
\end{enumerate}
\end{lem}

The open subset $\V\subsetneq \P^1_\bfk$ is invariant under the action of $\rho(\Gamma)$; hence $\rho(\Gamma)$ 
 is finite by Property~{\FW} and Lemma~\ref{lem:pgl2-fw}. 
Let $\Gamma_0$ be the kernel of this morphism. Let $\varphi\colon X\dasharrow \P^1_\bfk\times \P^1_\bfk$ be
a birational map that conjugates the fibration $\pi$ to the first projection $\tau\colon \P^1_\bfk\times \P^1_\bfk\to \P^1_\bfk$.
Then,  $\Gamma_0$ is conjugate to a subgroup 
of $\PGL_2(\bfk(x))$ acting on $\P^1_\bfk\times \P^1_\bfk$ by linear projective transformations of the fibers of $\tau$.
From Corollary~\ref{coro:Bass-k(C)}, a new conjugacy by an element of  $\PGL_2(\bfk(x))$ changes $\Gamma_0$ in
 an infinite subgroup of $\PGL_2(\oZ)$. 
Then, as in Sections~\ref{par:non-rational-surfaces} 
and~\ref{par:auto-rational-surfaces} we conclude that $\Gamma$ becomes a subgroup of $\PGL_2(\oZ)\times \PGL_2(\oZ)$, 
with a finite projection on the first factor.

\begin{pro}
Let $\Gamma$ be an infinite group with Property {\FW}, with $\Gamma\subset \Aut(\U)$, and $\U\subset Z$  as in case (4)
of Proposition~\ref{pro:Gizatullin-boundary}.
There exists a birational map $\psi\colon Z\dasharrow \P^1_\bfk\times \P^1_\bfk$ that conjugates 
$\Gamma$ to a subgroup of $\PGL_2(\overline{\Z})\times \PGL_2(\oZ)$, 
with a finite projection on the first factor.
\end{pro}

\subsection{Completions by zigzags}\label{par:completion-zigzag}

Two cases remain to be studied: $\partial Z$ can be a chain of rational curves (a zigzag in Gizatullin's terminology)
or a cycle of rational curves (a loop in Gizatullin's terminology). Cycles are considered in Section~\ref{par:Cycles-Thompson}. 
In this section, we rely on difficult results of  Danilov and Gizatullin to treat the case of chains of rational curves (i.e. case (3)
in Proposition~\ref{pro:Gizatullin-boundary}). Thus, in this section
\begin{enumerate}
\item[(i)] $\partial X$ is a chain of smooth rational curves $E_i$
\item[(ii)] $\U=X\smallsetminus \partial X$ is an affine surface (singularities are  allowed)
\item[(iii)] every irreducible component $E_i$ is contracted to a point of $\partial X$ by at least one element 
of $\Gamma\subset \Aut(\U) \subset \Bir(X)$.
\end{enumerate}

In \cite{Danilov-Gizatullin:I,Danilov-Gizatullin:II}, Danilov and Gizatullin introduce a set of ``standard completions'' of the affine surface $\U$. 
As in Section~\ref{par:constraints-boundary}, a
completion (or more precisely a ``marked completion'') is an embedding $\iota\colon \U\to Y$ into a
complete surface such that $\partial Y=Y\smallsetminus \iota(\U)$ is a curve (this boundary curve
may be reducible). Danilov and Gizatullin only consider completions for which $\partial Y$ is a chain 
of smooth rational curves and $Y$ is smooth in a neighborhood of $\partial Y$; the surface $X$ provides such a completion. 
Two  completions $\iota\colon \U\to Y$ and $\iota'\colon \U\to Y'$ are isomorphic if the birational map
$\iota'\circ \iota^{-1}\colon Y\to Y'$ is an isomorphism; in particular, the boundary curves are identified by this isomorphism.
The group $\Aut(\U)$ acts by pre-composition on the set of isomorphism classes of (marked) completions. 

Among all possible completions, Danilov and Gizatullin distinguish a class of ``standard (marked) completions'', for which 
we refer to \cite{Danilov-Gizatullin:I} for a definition. There are elementary links (corresponding to certain birational mappings $Y\dasharrow Y'$)
between standard completions, and one can construct a graph $\Delta_\U$ whose vertices are standard completions;
there is an edge between two completions if one can pass from one to the other by an elementary link. 

\begin{eg}
A completion is $m$-standard, for some $m\in \Z$, if the boundary curve $\partial Y$ is a chain of $n+1$ consecutive
rational curves $E_0$, $E_1$, $\ldots$, $E_n$ ($n\geq 1$) such that 
\[
E_0^2=0, \;\;  E_1^2=-m, \; {\text{ and }} \; E_i^2=-2\;\;  \text{if}\;\;  i\geq 2.
\] 
Blowing-up the intersection point $q=E_0\cap E_1$, one creates a new chain starting by $E_0'$ with $(E_0')^2=-1$;
blowing down $E_0'$, one creates a new $(m+1)$-standard completion. This is one of the elementary links.
\end{eg}

Standard completions are defined by constraints on the self-intersections of the components $E_i$. Thus, the action of $\Aut(\U)$ 
on completions permutes the standard completions; this action determines a morphism from $\Aut(\U)$ to the group of
isometries (or automorphisms) of the graph $\Delta_\U$ (see \cite{Danilov-Gizatullin:I}):
\[
\Aut(\U)\to \Iso(\Delta_\U).
\]
\begin{thm}[Danilov and Gizatullin, \cite{Danilov-Gizatullin:I, Danilov-Gizatullin:II}] The graph $\Delta_\U$ of all isomorphism classes of 
standard completions of $\U$ is a tree.
The group $\Aut(\U)$ acts by isometries of this tree. The stabilizer of a vertex $\iota\colon \U\to Y$ is the subgroup 
$G(\iota)$ of automorphisms of the complete surface $Y$ that fix the curve $\partial Y$. This group is an algebraic subgroup of $\Aut(Y)$.
\end{thm}

The last property means that $G(\iota)$ is an algebraic group that acts algebraically on $Y$. It coincides
with the subgroup of $\Aut(Y)$ fixing the boundary $\partial Y$; the fact that it is algebraic follows from 
the existence of a  $G(\iota)$-invariant, big and nef divisor which is supported on $\partial Y$ (see the last sentence
of Proposition~\ref{pro:Gizatullin-boundary}).  
The crucial assertion in this theorem is that $\Delta_\U$ is a simplicial tree (typically,   infinitely many edges emanate from each vertex).  
There are sufficiently many links 
 to assure connectedness, but not too many in order to prevent the existence of 
 cycles in the graph $\Delta_\U$. 

\begin{cor}
If $\Gamma$ is a subgroup of $\Aut(\U)$ that has the fixed point property on trees, then $\Gamma$ is contained 
in $G(\iota)\subset \Aut(Y)$  for some completion $\iota\colon \U\to Y$. 
\end{cor} 

If $\Gamma$ has Property {\FW}, it has Property {(FA)} (see Section~\ref{par:Bass}). Thus, if it acts by automorphisms on $\U$, 
$\Gamma$ is conjugate to the subgroup $G(\iota)$ of $\Aut(Y)$, for some zigzag-completion $\iota\colon \U\to Y$. 
Theorem~\ref{thm:classification-virtually-autom} of Section~\ref{par:auto-rational-surfaces} 
implies that the action of $\Gamma$ on the initial surface $X$ is conjugate to a regular action on 
$\P^2_\bfk$, $\P^1_\bfk\times \P^1_\bfk$ or $\Hirz_m$, for some Hirzebruch surface~$\Hirz_m$.
This action preserves a curve, namely the image of the zigzag into the surface $Y$. The following examples
list all possibilities, and conclude the proof of Theorem~B in the case of zigzags (i.e. case (3)
in Proposition~\ref{pro:Gizatullin-boundary}).

\begin{eg}
Consider the projective plane $\P^2_\bfk$, together with an infinite subgroup $\Gamma\subset \Aut(\P^2_\bfk)$ 
that preserves a curve $C$ and has Property {\FW}. Then, $C$ must be a smooth rational curve: either
a line, or a smooth conic. Indeed, if the genus of $C$ is positive, or if $C$ is rational but is not smooth, 
then the action of $\Gamma$ on $C$ factors through a finite quotient of $\Gamma$ (see Lemma~\ref{lem:pgl2-fw}); 
but then the image of $\Gamma$ in $\Aut(\P^2_\bfk)$ would be virtually solvable, hence finite by Property {\FW}.  Now, if $C$ is the line ``at infinity'', then $\Gamma$ acts by affine transformations 
on the affine plane $\P^2_\bfk\smallsetminus C$. If $C$ is the conic $x^2+y^2+z^2=0$, $\Gamma$ becomes 
a subgroup of ${\sf{PO}}_3(\bfk)$.
\end{eg}

\begin{eg}
When $\Gamma$ is a subgroup of $\Aut(\P^1_\bfk\times \P^1_\bfk)$ that preserves a curve $C$ and has Property {\FW}, 
then $C$ must be a smooth curve because $\Gamma$ has no finite orbit (Lemma~\ref{lem:pgl2-fw}). Similarly,
the two projections $C\to \P^1_\bfk$ being equivariant with respect to the morphisms $\Gamma\to \PGL_2(\bfk)$, they 
have no ramification points. Thus, $C$ is a smooth rational curve, and its projections onto each factor are isomorphisms. 
In particular, the action of $\Gamma$ on $C$ and on each factor are conjugate. These conjugacies show
that $\Gamma$  is conjugate to a diagonal embedding
\[
\gamma\in \Gamma\;  \mapsto \; (\rho(\gamma),\rho(\gamma)) \in \PGL_2(\bfk)\times \PGL_2(\bfk).
\]
\end{eg}

\begin{eg}
Similarly, the group $\SL_2(\bfk)$ acts on the Hirzebruch surface $\Hirz_m$, preserving the zero section of the 
fibration $\pi\colon \Hirz_m \to \P^1_\bfk$. This gives examples of groups with Property {\FW} acting on $\Hirz_m$
and preserving a big and nef curve $C$. 
\end{eg}

Starting with one of the above examples, one can  blow-up points on the invariant curve $C$, and then 
contract $C$, to get examples of zigzag completions $Y$ on which $\Gamma$ acts and contracts the boundary $\partial Y$. 

\section{Birational transformations of surfaces II}\label{par:Cycles-Thompson}

In this section, $\U$ is a (normal, singular) affine surface with a completion $X$ by a cycle of~$\ell$ rational
curves. Every irreducible component $E_i$ of the boundary $\partial X=X\smallsetminus \U$ is contracted
by at least one automorphism of $\U$. Our goal is to classify subgroups $\Gamma$ of $\Aut(\U)\subset \Bir(X)$
that are infinite and have Property {\FW}: in fact, we shall show that no such group exists. This ends the proof of 
Theorem~B since the other possibilities of Proposition~\ref{pro:Gizatullin-boundary} have been dealt with in the previous section. 

\begin{rem}
The proof is based on the fact that $\Aut(\U)$ acts in a piecewise $\PGL(2,\Z)$ way on a circle whose rational 
points correspond to divisors at infinity in various compactifications of $\U$. To describe this action, our presentation is
similar to the one in \cite{Hubbard-Papadopol}. Another equivalent, more precise,  but slightly longer route is to consider the set of valuations on
the ring of regular functions on $\U$ which are centered on $\partial X$. The circle we are looking for corresponds
to a certain set of valuations with log-discrepancy $0$; this approach is described in a particular case in 
\cite{NBD};  to study the log-discrepancy in our context, one could refer to \cite{DJS} (in order to construct a regular $2$-form
on $\U$ with poles exactly along $\partial X$ after compactification). 
Also, we use both Farey and dyadic partitions of the circle because the Farey viewpoint is used by algebraic geometers, 
while  dyadic partitions are often used in group theory  (see \cite{Navas:Book}, \S 1.5); these are
just two equivalent viewpoints.
\end{rem}

\begin{eg}
Let $(\A_\bfk^1)^*$ denote the complement of the origin in the affine line $\A^1_\bfk$; it is isomorphic to 
the multiplicative group ${\mathbb{G}}_m$ over $\bfk$. The surface $(\A_\bfk^1)^*\times (\A_\bfk^1)^*$ 
is an open subset in $\P^2_\bfk$ whose boundary is the triangle of coordinate lines $\{[x:y:z];\; xyz=0\}$.
Thus, the boundary is a cycle of length $\ell=3$.
The group of automorphisms of  $(\A_\bfk^1)^*\times (\A_\bfk^1)^*$ is the semi-direct product 
$
\GL_2(\Z)\ltimes ({\mathbb{G}}_m(\bfk)\times {\mathbb{G}}_m(\bfk));
$
it does not contain any infinite group with Property {\FW}.
\end{eg}

\subsection{Resolution of indeterminacies}

Let us order cyclically the irreducible components $E_i$ of $\partial X$, so that $E_i\cap E_j\neq \emptyset$ 
if and only if $i-j=\pm 1 (\mathrm{mod}\,\ell)$. Blowing up finitely many singularities of $\partial X$, 
we may assume that $\ell=2^m$ for some integer $m\geq 1$; in particular, every curve $E_i$ is smooth.
(With such a modification, one may a priori create irreducible components of $\partial X$ that are not contracted
 by the group~$\Gamma$.)

\begin{lem}\label{lem:Indet-Cyclic}
Let $f$ be an automorphism of $\U$ and let $f_X$ be the birational extension of $f$ to the surface $X$. Then 
\begin{enumerate}
\item Every indeterminacy point of $f_X$ is a singular point of $\partial X$, i.e.\ one of the intersection points
$E_i\cap E_{i+1}$.

\item Indeterminacies of $f_X$ are resolved by inserting chains of rational curves.
\end{enumerate}
\end{lem}

Property (2) means that there exists a resolution of the indeterminacies of $f_X$, given by two birational 
morphisms $\epsilon\colon Y \to X$ and $\pi\colon Y\to X$ with $f\circ \epsilon = \pi$, such that 
$\pi^{-1}(\partial X)=\epsilon^{-1}(\partial X)$ is a cycle of rational curves. Some of the singularities of $\partial X$
have been blown-up into chains of rational curves to construct $Y$.

\begin{figure}[h]
\input{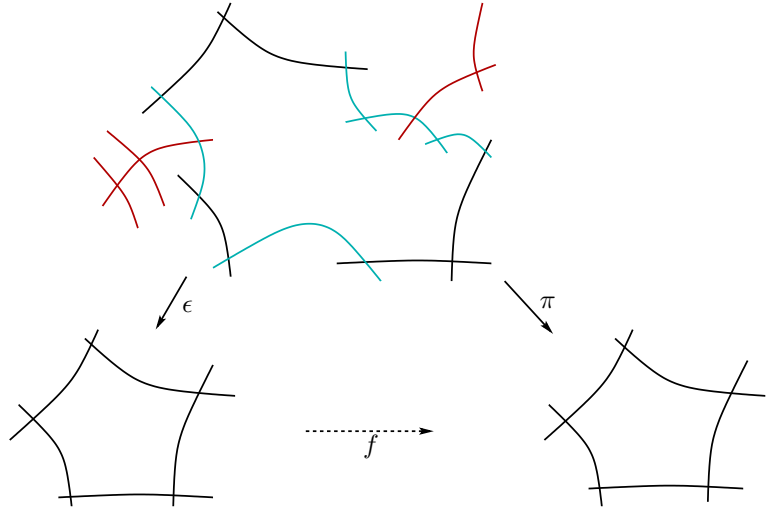}
\caption{{\small{A blow-up sequence creating two (red) branches. No branch of this type appears for minimal resolution.}}}
\end{figure}

\begin{proof}
Consider a minimal resolution of the indeterminacies of $f_X$. It is given by a finite sequence of blow-ups
of the base points of $f_X$, producing a surface $Y$ and two birational morphisms $\epsilon\colon Y \to X$ 
and $\pi\colon Y\to X$ such that $f_X=\pi\circ \epsilon^{-1}$. Since the indeterminacy points of $f_X$ are contained
in $\partial X$, all necessary blow-ups are centered on $\partial X$.

The total transform $F=\epsilon^*(\partial X)$ is a union of rational curves: it is made of a cycle, together with branches
emanating from it. One of the assertions (1) and (2) fails if and only if $F$ is not a cycle; in that case,
there is at least one branch.

Each branch is a tree of smooth rational curves, which may be blown-down onto a smooth point; indeed, these branches 
come from  smooth points of the main cycle that have been blown-up finitely many times. Thus, there is a birational 
morphism $\eta\colon Y\to Y_0$ onto a smooth surface $Y_0$ that contracts the branches (and nothing more). 

The morphism $\pi$ maps $F$ onto the cycle $\partial X$, so that  all branches of $F$ are contracted by $\pi$. Thus, 
both $\epsilon$ and $\pi$ induce (regular) birational morphisms $\epsilon_0\colon Y_0\to X$ and $\pi_0\colon Y_0\to X$. 
This contradicts the minimality of the resolution. 
\end{proof}

Let us introduce a family of surfaces 
\[
\pi_k\colon X_k\to X.
\]
First, $X_1=X$ and $\pi_1$ is the identity map. Then, $X_2$ is obtained by blowing-up the~$\ell$ singularities
of $\partial X_1$; $X_2$ is a compactification of $\U$ by a cycle $\partial X_2$ of $2\ell=2^{m+1}$ smooth rational curves. Then, 
$X_3$ is obtained by blowing up the singularities of $\partial X_2$, and so on. In particular, $\partial X_k$ is a cycle of $2^{k-1}\ell = 2^{m+k-1}$
curves. 

Denote by $\Dual_k$ the {\bf{dual graph}} of $\partial X_k$: vertices of $\Dual_k$ correspond to irreducible 
components $E_i$ of $\partial X_k$ and edges to intersection points $E_i\cap E_j$. A simple blow-up 
(of a singular point) modifies both $\partial X_k$ and $\Dual_k$ locally as follows

\begin{figure}[h]
\input{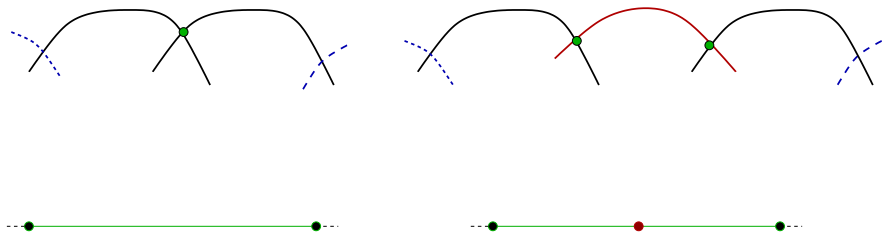}
\caption{{\small{Blowing-up one point.}}}\label{fig:2}
\end{figure}

The group $\Aut(\U)$ acts on $\Hyp(X)$ and Lemma~\ref{lem:Indet-Cyclic} shows that its action stabilizes the subset 
$\B$ of $\Hyp(X)$ defined as
\[
\B=\left\{C\in \Hyp(X): \; \exists k\geq 1, C\; \,  \text{is an irreducible component of } \; \partial X_k \right\}.
\]
In what follows, we shall parametrize $\B$ in two distinct ways by rational numbers.  

\subsection{Farey and dyadic parametrizations}\label{fareydy}

Consider an edge of the graph $\Dual_1$, and identify this edge with the unit interval $[0,1]$. 
Its endpoints correspond to two adjacent components $E_i$ and $E_{i+1}$ of $\partial X_1$, 
and the edge corresponds to their intersection~$q$. Blowing-up $q$ creates a new vertex (see Figure~\ref{fig:2}).  The edge
is replaced by two adjacent edges of $\Dual_2$ with a common vertex corresponding to the
exceptional divisor and the other vertices corresponding to (the strict transforms of) $E_i$ and
$E_{i+1}$; we may identify this part of $\Dual_2$ with the segment $[0,1]$, the three vertices 
with $\{0, 1/2, 1\}$, and the two edges with $[0,1/2]$ and $[1/2,1]$. 

Subsequent blow-ups may be organized in two different ways by using either a dyadic or a Farey algorithm (see Figure~\ref{fig:3}). 

In the dyadic algorithm, the vertices are labelled by dyadic numbers $m/2^k$. The vertices of $\Dual_{k+1}$ coming 
from an initial edge $[0,1]$ of $\Dual_1$ are the $2^k+1$ points
$\{n/2^{k}; \; 0\leq n \leq 2^k\}$ of the segment $[0,1]$. We denote by $\Dya(k)$ the set of dyadic numbers $n/2^k\in [0,1]$; thus, 
$\Dya(k)\subset \Dya(k+1)$.
We say that an interval $[a,b]$ is a {\bf{standard dyadic}} interval if $a$ and $b$ are two consecutive numbers in 
$\Dya(k)$ for some~$k$.

In the Farey algorithm, the vertices correspond to rational numbers $p/q$. 
Adjacent vertices of $\Dual_k$ coming from the initial segment $[0,1]$
correspond to pairs of rational numbers $(p/q,r/s)$ with $ps-qr=\pm 1$; two adjacent vertices of $\Dual_k$
 give birth to a new, middle vertex in $\Dual_{k+1}$: this middle vertex is $(p+r)/(q+s)$
(in the dyadic algorithm, the middle vertex is the ``usual'' euclidean middle). We shall say that an interval $[a,b]$ is a
{\bf{standard Farey}} interval if $a=p/q$ and $b=r/s$ with $ps-qr=-1$. We denote by $\Far(k)$ the finite set of rational 
numbers $p/q\in [0,1]$ that is given by the $k$-th step of Farey algorithm; thus, $\Far(0)=\{0,1\}$ and $\Far(k)$
is a set of $2^{k}+1$ rational numbers $p/q$ with $0\leq p\leq q$. (One can check that $1\leq q\leq \Fib(k+2)$, with 
$\Fib(k)$ the $k$-th term in the Fibonacci sequence $\Fib(0)=0$, $\Fib(1)=1$, $\Fib(m+1)=\Fib(m)+\Fib(m-1)$.)

\begin{figure}[h]
\input{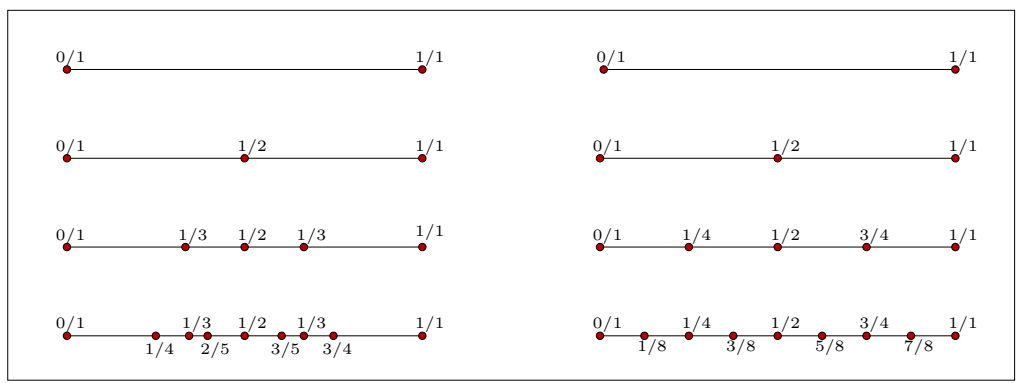}
\caption{{\small{On the left, the Farey algorithm. On the right, the dyadic one. Here $k=0$ (top), to $k=3$ (bottom).}}}\label{fig:3}
\end{figure}

By construction, the graph $\Dual_1$ has $\ell=2^m$ edges. The edges of $\Dual_1$ are in one-to-one correspondence with the singularities $q_j$ of $\partial X_1$. Each edge determines a subset $\B_j$ of $\B$; the elements
of $\B_j$ are the 
curves $C\subset \partial X_k$ ($k\geq 1$) such that $\pi_k(C)$ contains the singularity $q_j$ determined by the edge.
Using the dyadic algorithm (resp. Farey algorithm), the elements of $\B_j$  are in one-to-one correspondence with 
dyadic (resp. rational) numbers in $[0,1]$. Gluing these segments cyclically one gets a circle $\SS^1$, together with a nested 
sequence of subdivisions in $\ell$, $2\ell$, $\ldots$, $2^{k-1}\ell$, $\ldots$ intervals; each interval is 
a standard dyadic (or Farey) interval of one of the initial edges.

Since there are $\ell=2^m$ initial edges, we may identify the graph $\Dual_1$ 
with the circle $\SS^1=\R/\Z=[0,1]/_{0\simeq 1}$ and the initial vertices
with the dyadic numbers in $\Dya(m)$ modulo $1$ (resp. the elements of $\Far(m)$ modulo $1$). 
The vertices of $\Dual_k$ are in one to one correspondence with the dyadic numbers in $\Dya(k+m-1)$.

\begin{rem}\label{rem:Farey-Thompson}
\vspace{0.1cm}
\noindent{(a).--} By construction, the interval $[p/q,r/s]\subset [0,1]$ is a standard Farey interval if and only if $ps-qr=-1$, iff it is delimited  by two adjacent
elements of $\Far(m)$ for some $m$.

\vspace{0.1cm}
\noindent{(b).--} If $h\colon [x,y]\to [x',y']$ is a homeomorphism between two standard Farey intervals mapping
rational numbers to rational numbers and standard Farey intervals to standard Farey intervals, then $h$ is the restriction 
to $[x,y]$ of a unique linear projective transformation with integer coefficients: 
\[
h(t)=\frac{at+b}{ct+d}, \; \text{ for some element }\;  \left( \begin{array}{cc} a & b \\ c & d \end{array}\right) \; \text{ of }\; \PGL_2(\Z).
\]

\vspace{0.1cm}
\noindent{(c).--} Similarly, if $h$ is a homeomorphism mapping 
standard dyadic intervals to intervals of the same type, then $h$ is the
restriction of an affine dyadic map
\[
h(t)=2^mt+\frac{u}{2^n}, \; {\text{ with }} m, n \in \Z.
\]
\end{rem}

In what follows, we denote by $\Tho$ the group of self-homeomorphisms of $\SS^1=\R/\Z$ that are piecewise
$\PGL_2(\Z)$ mapping with respect to a finite decomposition of the circle in standard Farey  intervals $[p/q,r/s]$. 
In other words, if $f$ is an element of $\Tho$, there are two partitions of the circle into consecutive intervals 
$I_i$ and $J_i$ such that the $I_i$ are intervals with rational endpoints, $h$ maps $I_i$ to $J_i$, and the
restriction $f\colon I_i\to J_i$ is the restriction of an element of $\PGL_2(\Z)$ (see \cite{Navas:Book}, \S 1.5.1). 

\begin{thm}
Let $\U$ be an affine surface with a compactification $\U \subset X$ such that $\partial X:=X\smallsetminus \U$
is a cycle of smooth rational curves. In the Farey parametrization of the set $\B\subset \Hyp(X)$ of boundary curves, 
the group $\Aut(\U)$ acts on $\B$ as a subgroup of $\Tho$.
\end{thm}

\begin{rem}
There is a unique orientation preserving self-homeomorphism of the circle that maps $\Dya(k)$ to $\Far(k)$ for every $k$.
This self-homeomorphism conjugates $\Tho$ to the group $\Thom$ of self-homeomorphisms of the circle that are piecewise
affine with respect to a dyadic decomposition of the circle, with slopes in $\pm 2^\Z$, and with translation parts in $\Z[1/2]$. 
Using the parametrization of $\B$ by dyadic numbers, the image of $\Aut(\U)$ becomes a subgroup of $\Thom$.
\end{rem}

\begin{proof} Lemma~\ref{lem:Indet-Cyclic}    is the main ingredient. 
Consider the action of the group $\Aut(\U)$ on the set $\B$. Let $f$ be an element of $\Aut(\U)\subset \Bir(X)$. Consider an irreducible 
curve $E\in \B$, and denote by $F$ its image: $F=f_\bullet(E)$ is an element of $\B$ by Lemma~\ref{lem:Indet-Cyclic}. 
There are integers $k$ and $l$ such that $E\subset \partial X_k$ and
$F\subset \partial X_l$. Replacing $X_k$ by a higher blow-up $X_m\to X$, we may assume that $f_{lm} :=\pi_l^{-1}\circ f\circ \pi_m$ is regular on a neighborhood of the curve $E$ (Lemma~\ref{lem:Indet-Cyclic}). Let $q_k$ be one of the two singularities of $\partial X_m$ that are contained in $E$, 
and let $E'$ be the second irreducible component of $\partial X_m$ containing $q$. If $E'$ is blown down by $f_{lm}$, its image is
one of the two singularities of $\partial X_l$ contained in $F$ (by Lemma~\ref{lem:Indet-Cyclic}). 
Consider the smallest integer $n\geq l$ such that $\partial X_n$ contains 
the strict transform $F'=f_\bullet(E')$; in $X_n$, the curve $F'$ is adjacent to the strict transform of $F$ (still denoted $F$), and $f$ is a local
isomorphism from a neighborhood of $q$ in $X_m$ to a neighborhood of $q':=F\cap F'$ in $X_n$. 

Now, if one blows-up $q$, the exceptional divisor $D$ is mapped by $f_\bullet$ to the exceptional divisor $D'$ obtained by a simple blow-up of $q$: 
$f$ lifts to a local isomorphism from a neighborhood of $D$ to a neighborhood of $D'$, the action from $D$ to $D'$ being given by the
differential $df_q$. The curve $D$ contains two singularities of $\partial X_{m+1}$, which can be blown-up too: again, $f$ lifts to a local isomorphism
if one blow-ups the singularities of $\partial X_{n+1}\cap D'$. We can repeat this process indefinitely. Let us now phrase this remark differently. 
The point $q$ determines an edge of $\Dual_m$, hence a standard Farey interval $I(q)$. The point $q'$ determines an edge of $\Dual_n$, 
hence another standard Farey interval $I(q')$. Then, the points of $\B$ that are parametrized by rational numbers in $I(q)$ are mapped
by $f_\bullet$ to rational numbers in $I(q')$ and this map respects the Farey order: if we identify $I(q)$ and $I(q')$ to $[0,1]$, 
$f_\bullet$ is the restriction of a monotone map that sends $\Far(k)$ to $\Far(k)$ for every $k$. 
Thus, on $I(q)$, $f_\bullet$ is the restriction of a linear projective transformation with integer coefficients 
(see Remark~\ref{rem:Farey-Thompson}-(b)). This shows that $f_\bullet$ is an element of $\Tho$.
\end{proof}

\subsection{Conclusion}

Consider the group $\Thom^*$ of self-homeomorphisms of the circle $\SS^1=\R/\Z$ 
that are piecewise affine with respect to a finite partition of $\R/\Z$ into dyadic 
intervals $[x_i,x_{i+1}[$ with $x_i$ in $\Z[1/2]/\Z$ for every $i$, and satisfy 
$
h(t)=2^{m_i}t+a_i
$
with $m_i\in \Z$ and $a_i\in \Z[1/2]$ for every $i$. This group is known as the Thompson group 
of the circle, and is isomorphic to the group $\Tho^*$ of orientation-preserving self-homeomorphisms
in $\Tho$ (defined in \S\ref{fareydy}).

\begin{thm}[Farley, Hughes \cite{Farley:2003, hughesfarley}]\label{thm:Farley-Navas}
Every subgroup of the Thompson group $\Thom^*$ (and hence of $\Tho^*$) with Property {\FW} is a finite cyclic group. 
\end{thm}
Indeed fixing a gap in an earlier construction of Farley \cite{Farley:2003}\footnote{The gap in Farley's argument lies in Prop.\ 2.3 and Thm.\ 2.4 of \cite{Farley:2003}.}, Hughes proved \cite{hughesfarley} that $\Tho$ has Property PW, in the sense that it admits a commensurating action whose associated length function is a proper map (see also Navas' book \cite{Navas:Book}). This implies the conclusion, because every finite group of orientation-preserving self-homeomorphisms of the circle is cyclic.

Thus, if $\Gamma$ is a subgroup of $\Aut(\U)$ with Property {\FW}, it contains a finite index subgroup $\Gamma_0$
that acts trivially on the set $\B\subset \Hyp(X)$. This means that $\Gamma_0$ extends as a group of automorphisms
of $X$ fixing the boundary $\partial X$. Since $\partial X$ supports a big and nef divisor, $\Gamma_0$ contains a finite
index subgroup $\Gamma_1$ that is contained in $\Aut(X)^0$. 

Note that $\Gamma_1$ has Property {\FW}  because it is a finite index subgroup of $\Gamma$. It preserves
every irreducible component of the boundary curve $\partial X$, as well as its singularities. As such, it must act trivially 
on $\partial X$. When we apply Theorem~\ref{thm:classification-virtually-autom} to~$\Gamma_1$, 
the conjugacy $\varphi\colon X\to Y$ can not contract $\partial X$, because the boundary supports
an ample divisor. Thus, $\Gamma_1$ is conjugate to a subgroup of $\Aut(Y)$
that fixes a curve pointwise. This is not possible if $\Gamma_1$ is infinite (see Theorem~\ref{thm:classification-virtually-autom} and the remarks
following it).

We conclude that $\Gamma$ is finite in case (2) of Proposition~\ref{pro:Gizatullin-boundary}. 

\section{Birational actions of $\SL_2(\Z[\sqrt{d}])$}\label{scorcorSL2}

We develop here Example \ref{applx}. If $\bfk$ is an algebraically closed field of characteristic $0$, therefore containing ${\overline{\Q}}$, 
we denote by  $\sigma_1$ and $\sigma_2$ the distinct embeddings of $\Q(\sqrt{d})$ into $\bfk$. Let $j_1$ and $j_2$ be the resulting embeddings of $\SL_2(\Z[\sqrt{d}])$ into $\SL_2(\bfk)$, and $j=j_1\times j_2$ the compound embedding into $\mathsf{G}=\SL_2(\bfk)\times\SL_2(\bfk)$.

\begin{thm}\label{tSL2}
Let $\Gamma$ be a finite index subgroup of  $\SL_2(\Z[\sqrt{d}])$. Let $X$ be an irreducible projective surface over an algebraically closed field $\bfk$. Let $\alpha:\Gamma\to\Bir(X)$ be a homomorphism with infinite image. Then $\bfk$ has characteristic zero, and there exist a finite index subgroup 
$\Gamma_0$ of $\Gamma$ and a birational map $\varphi:Y\dasharrow X$ such that
\begin{enumerate}
\item $Y$ is the projective plane $\P^2$, a Hirzebruch surface $\F_m$, or $C\times\P^1$ for some curve $C$;
\item $\varphi^{-1}\alpha(\Gamma)\varphi\subset\Aut(Y)$;
\item there is a unique algebraic homomorphism $\beta:\mathsf{G}\to\Aut(Y)$ such that 
$
\beta(j(\gamma))=\varphi^{-1}\alpha(\gamma)\varphi
$
for every $\gamma\in\Gamma_0$.
\end{enumerate}
\end{thm}

To prove this result, assume first that  $\bfk$ has positive characteristic. Theorem~B ensures that $Y$ is the projective plane, and the $\Gamma$-action is given by a homomorphism into $\PGL_3(\bfk)$. Then remark that every homomorphism $\tau :\Gamma\to\GL_n(\bfk)$ has finite image; indeed, it is well-known that $\GL_n(\bfk)$ has no infinite order distorted element: elements of infinite order have some transcendental eigenvalue and the conclusion easily follows. Since $\Gamma$ has an exponentially distorted cyclic subgroup, the kernel of $\tau$ is infinite, and by the Margulis normal subgroup theorem the image of $\tau$ is finite.

Now, assume that the characteristic of $\bfk$ is $0$. From
Theorem B, Assertions (1) and (2) are satisfied. 
If $Y$ is $\P^2$, $\P^1\times \P^1$, or a Hirzebruch surface $\F_m$, then $\Aut(Y)$ is a linear algebraic group. 
If $Y$ is a product $C\times \P^1$, with $g(C)\geq 1$, the projection onto $C$ gives a $\Gamma$-equivariant morphism; since
$g(C)\geq 1$, the automorphism group of $C$ is virtually abelian, and a finite index subgroup $\Gamma_1$ of $\Gamma$
acts trivially on $C$. Thus, the action of $\Gamma_1$ on $Y$  preserves the projection 
onto $\P^1$ and  acts via an embedding into the linear algebraic group $\Aut(\P^1)=\PGL_2(\bfk)$. 
Then, the proof of Theorem \ref{tSL2} follows from the next lemma.

\begin{lem}
Let $\bfk$ be a field containing $\Q(\sqrt{d})$.
Consider the compound embedding $j$ of $\SL_2(\Z[\sqrt{d}])$ into $G=\SL_2(\bfk)\times\SL_2(\bfk)$.
For every linear algebraic group $H$ and homomorphism $f:\SL_2(\Z[\sqrt{d}])\to H(\bfk)$, there exists a unique homomorphism $\tilde{f}:G\to H$ of $\bfk$-algebraic groups such that the homomorphisms $f$ and $\tilde{f}\circ j$ coincide on some finite index subgroup of $\Gamma$.
\end{lem}

\begin{proof}[Sketch of Proof]
The uniqueness is a consequence of Zariski density of the image of~$j$. Let us prove the existence. Zariski density allows to reduce to the case when $H=\SL_n$. In the case $\bfk=\R$, one first remarks that the image of $\SL_2(\Z[\sqrt{d}])$ in $\SL_n(\R)$ is not contained in a compact group because $\SL_2(\Z[\sqrt{d}])$ contains exponentially distorted elements. Then, Margulis' superrigidity and the fact that every continuous real representation of $\SL_2(\R)$ is algebraic prove the lemma. The case of fields containing $\R$ immediately follows, and in turn it follows for subfields of overfields of $\R$ (as soon as they contain $\Q(\sqrt{d})$).
\end{proof}

\section{Open problems}

\begin{que}
Let $\Gamma$ be a group with Property {\FW}. Is every birational action of $\Gamma$ regularizable ? Here regularizable is defined in the same way as pseudo-regularizable, but assuming that the action on $\U$ is by automorphisms (instead of pseudo-automorphisms).
\end{que}

A particular case is given by Calabi-Yau varieties (simply connected complex projective manifolds $X$ with trivial canonical bundle
and $h^{k,0}(X)=0$ for $0< k<\dim(X)$). For such a variety, $\Bir(X)$ coincides with $\Psaut(X)$. 
One can then ask (1) whether every  subgroup $\Gamma$ of $\Psaut(X)$ with property {\FW} is regularizable on some 
birational model $Y$ of $X$ (without restricting the action to a dense open subset), and (2) what are the possibilities 
for such a group~$\Gamma$.

\begin{que}\label{ruled3}
For which irreducible projective varieties $X$ 
\begin{enumerate}
\item\label{bi1} $\Bir(X)$ does not transfix $\Hy(X)$?
\item\label{bi2} some finitely generated subgroup of $\Bir(X)$ does not transfix $\Hy(X)$?
\item\label{bi3} some cyclic subgroup of $\Bir(X)$ does not transfix $\Hy(X)$.
\end{enumerate}
\end{que}

We have the implications: $X$ is ruled $\Rightarrow$  (\ref{bi3}) $\Rightarrow$ (\ref{bi2}) $\Rightarrow$ (\ref{bi1}). In dimension $2$, we have: ruled $\Leftrightarrow$ (\ref{bi1}) $\Rightarrow\!\!\!\!\!\!\!\!\!/$ (\ref{bi2}) $\Leftrightarrow$ (\ref{bi3}) (see \S\ref{surf_birt}). It would be interesting to find counterexamples to these equivalences in higher dimension, and settle each of the problems raised in Question \ref{ruled3} in dimension $3$.

\vspace{0.1cm}

The group of affine transformations of $\A^3_\C$ contains $\SL_3(\C)$, and this group contains many subgroups with 
Property {\FW}. For surfaces, Theorem~B shows that groups of birational transformations with Property~{\FW}
are contained in algebraic groups, up to conjugacy. The following question asks whether this type of theorem may hold
for $\Aut(\A^3_\C)$.  

\begin{que} 
Does there exist an infinite subgroup of $\Aut(\A^3_\C)$ with Property~{\FW} that is not conjugate to 
a group of affine transformations of $\A^3_\C$ ?
\end{que}

Recall that a length function $\ell$ on a group $G$ is a function $\ell \colon G\to \R_+$ such that $\ell(g)=0$ if and only if $g$ is
the neutral element, 
$\ell(g)=\ell(g^{-1})$, and $\ell(gh)\leq \ell(g)+\ell(h)$ 
for every pair of elements $g$ and $h$ in $G$. A length function is quasi-geodesic if there exists $M>0$ such that for every $n\ge 1$ and every $g\in G$ with $\ell(g)\le n$, there exist $1=g_0$, $g_1,\dots,$ $g_n=g$ in $G$ such that $\ell(g_{i-1}^{-1}g_i)\le M$ for all $i$. Equivalently $G$, endowed with the distance $(g,h)\mapsto\ell(g^{-1}h)$, is quasi-isometric to a connected graph.

\begin{que}
Given an irreducible variety $X$, is the length function 
\[
g\in \Bir(X) \mapsto |\Hy(X)\triangle g\Hy(X)|
\] 
quasi-geodesic? In particular, what about $X=\P^2$ and the Cremona group $\Bir(\P^2)$?
\end{que}

\bibliographystyle{plain}
\bibliography{references-fw}

\end{document}